\documentclass[11pt, fleqn, reqno]{amsart}
\usepackage{amsmath,amsfonts,amssymb,amsthm,xcolor}
\usepackage{epsfig}

\usepackage{bbm}
\usepackage{enumerate}
\usepackage{amstext}
\usepackage{mathrsfs}
\usepackage{xspace}
\usepackage{amsfonts}
\usepackage{amsmath}
\usepackage{amssymb}
\usepackage{amstext}
\usepackage{amsthm}       %proof environment :)
\usepackage{xspace}
\usepackage[active]{srcltx}

\newcommand{\Bor}{\mathscr{B}}

\newcommand{\vBUC}{\big(\D{BUC}(\R^N)\big)^m}
\newcommand{\tvBUC}{\big(\D{BUC}(\clTcyl)\big)^m}
\newcommand{\vLip}{\big(\D{Lip}(\R^N)\big)^m}
\newcommand{\tvLip}{\big(\D{Lip}(\clTcyl)\big)^m}

\newcommand{\curves}{\D{C}\big([0,T];\R^N\big)}

\newcommand{\cyl}{(0,+\infty)\times\R^N}
\newcommand{\ccyl}{[0,+\infty)\times\R^N}
\newcommand{\clcyl}{\R_+\times\R^N}
\newcommand{\Tcyl}{(0,T)\times\R^N}
\newcommand{\clTcyl}{[0,T]\times\R^N}

\newcommand{\CC}{\mathcal C}
\newcommand{\I}{\mathcal I}

\newcommand{\ind}{\{1,\dots,m\}}

\newcommand{\EE}{\mathbb E}
\newcommand{\F}{\mathcal F}
\newcommand{\FF}{\mathbb F}
\newcommand{\boldg}{\mathbf{g}}
\newcommand{\R}{\mathbb R}

\newcommand{\N}{\mathbb N}
\newcommand{\PP}{\mathbb{P}}
\newcommand{\parts}{\mathscr{P}}
\newcommand{\Q}{\mathbb Q}

\newcommand{\restr}{\D{\Large $\llcorner$}}

\newcommand{\eps}{\varepsilon}

\newcommand{\uu}{\mathbf u}
\newcommand{\vv}{\mathbf v}
\newcommand{\ww}{\mathbf w}
\newcommand{\aaa}{\mathbf a}
\newcommand{\bbb}{\mathbf b}

\newcommand{\1}{\mathbbm1}

\newcommand{\e}{\D{e}}

\renewcommand{\S}{\mathcal{S}}
\newcommand{\leb}{\mathcal{L}}

\newcommand{\tagliato}{$\kern-5 mm -$}
\newcommand{\tagliat}{$\kern-4 mm -$}

\newcommand{\cchi}{\mbox{\large $\chi$}}

\newcommand{\D}[1]{\mbox{\rm #1}}
\newcommand{\dd}{\D{d}}

\newcommand{\om}{\omega}

\newcommand{\cv}[1]
{\, \displaystyle{\mathop{\longrightarrow}\limits_{#1}}\, }

\newtheorem{teorema}{Theorem}[section]
\newtheorem{prop}[teorema]{Proposition}
\newtheorem{lemma}[teorema]{Lemma}
\newtheorem{definition}[teorema]{Definition}
\newtheorem{cor}[teorema]{Corollary}

\newtheorem{guess}[teorema]{Remark}
\newtheorem{example}[teorema]{Example}

\newenvironment{oss}{\begin{guess} \begin{rm}}{\end{rm} \end{guess}}

\begin{document}
\date{\today}

\title{ Random Lax--Oleinik semigroups \\for Hamilton--Jacobi systems}
\author{Andrea Davini, Antonio Siconolfi \and Maxime Zavidovique}
\address{Dip. di Matematica, {Sapienza} Universit\`a di Roma,
P.le Aldo Moro 2, 00185 Roma, Italy}
\email{davini@mat.uniroma1.it}
\address{Dip. di Matematica, {Sapienza} Universit\`a di Roma,
P.le Aldo Moro 2, 00185 Roma, Italy}
\email{siconolf@mat.uniroma1.it}
\address{
IMJ (projet Analyse Alg\' ebrique), UPMC, 4, place Jussieu, Case
247, 75252 Paris C\' edex 5, France}
\email{zavidovique@math.jussieu.fr} 

\begin{abstract}
Following the random approach of \cite{SMT14},  we define a
Lax--Oleinik formula adapted  to evolutive weakly coupled systems of
Hamilton--Jacobi equations. It is reminiscent of the corresponding
 scalar  formula, with the relevant difference that it has a stochastic character since it  involves, loosely speaking, random
switchings between the various associated Lagrangians.
 We prove that the related value functions are viscosity solutions to the system, and establish existence of minimal random
 curves  under fairly general hypotheses. Adding  Tonelli like assumptions on the
 Hamiltonians, we  show differentiability  properties of such minimizers, and  existence of adjoint random
 curves. Minimizers and adjoint curves are  trajectories of a twisted generalized Hamiltonian dynamics.
\end{abstract}
\maketitle
\section*{Introduction}

The aim of the paper is  to define a Lax--Oleinik formula adapted to
evolutive weakly coupled Hamilton--Jacobi systems and study its main
properties. The system can be  written as
\begin{equation}\label{intro evolutive system}
\partial_t u_i+H_i(x,D_xu_i)+\big(B\uu(t,x)\big)_i=0\qquad\hbox{in
$(0,+ \infty)\times\R^N$}
\end{equation}
for $i\in\ind$, where $\uu =(u_1,
\cdots, u_m)$ is the unknown function, and the $H_i$ are unrelated
Hamiltonians satisfying rather standard conditions, see Section
\ref{sys}. The hypotheses taken on the coupling matrix $B=(b_{ij})$
correspond to suitable monotonicity properties of the equations  with respect to the
entries $u_j$, see Remark \ref{oss monotone system}. They are complemented by  a  degeneracy
condition requiring all the rows of $B$ to sum to 0, yielding that $-B$ is
generator of a semigroup of stochastic matrices.

It is worth pointing out the relevance of such a formula in the case
of a single equation. Besides providing a variational way to
represent viscosity solutions of related evolutionary or stationary
equations, it enters crucially into play in a variety of theoretical
constructions and problems. Just to give some examples of its range
of application, we mention that the Weak KAM Theory, as developed by
Fathi \cite{Fa}, is founded on the Lax--Oleinik formula. Bernard's
construction of regular subsolutions relies on a perturbation of a
suitable initial datum via alternate application of the two
conjugate Lax--Oleinik semigroups \cite{BernardC11}. The
 variational representation formulae for solutions of
Hamilton--Jacobi equations was exploited as a key tool to establish
several asymptotic results, such as homogenization in random media
\cite{Sou99, RezTa00}, large--time behavior of solutions \cite{Fa1,
DS, Ishii08}, selection principles in the ergodic approximation of
the Hamilton--Jacobi equation \cite{DFIZ}.

A dynamical interpretation of the system  setting is illuminating
and provides some insight on our method. At least when the $H_i$
satisfy Tonelli like regularity assumptions, the $m$ Hamiltonian
dynamics related to the $H_i$ can be viewed as  possible evolutions
of a system, with  coupling term providing  random switching
criteria. Randomness being governed by  the continuous--time Markov
chain with $-B$ as transition matrix. In this context, the  adapted
Lax--Oleinik formula is devised to  define the associated
expectation semigroup.

The pattern can be thought as a nonlinear version  of the so--called
random evolution, a topic initiated by Reuben Hersh at the end of
the sixties and pursued by several authors as Griego and Pinsky, see
\cite{GH, Pi}. The theory provides a mathematical frame to models
where evolving systems modify the mode of motion because of random
changes in the environment.

Mostly using a pure PDE  approach,  weakly coupled systems have been
recently widely investigated, in the stationary as well as in the
time--dependent version.  The main purpose being  to find parallels,
under the aforementioned degeneracy assumption on the coupling
matrix, with Weak KAM theory for scalar Eikonal equations.

 This
stream of research was initiated in \cite{CamLey}, where the authors
studied homogenization \` a la Lions--Papanicolaou--Varadhan
\cite{LPV}, and pursued   in \cite{leyetal} with the proof of time
convergence results for solutions of evolutive problems, under
hypotheses close to  \cite{NR}. Other outputs in this vein can be
found in a series of works including \cite{Ng,MitTran2}. The links
with weak KAM theory were further made precise by two of the authors
of the present paper (AD and MZ) in \cite{DavZav14} where, among
other things, an appropriate  notion of Aubry set for systems was
given and some relevant properties of it were generalized from the
scalar case. This study partially relies on the properties of the
semigroup associated to the evolutive system \eqref{intro evolutive
system}, but, due to the inability to provide a variational
representation for it, such properties are established purely by
means of PDE tools and viscosity solution techniques. Weak KAM
Theory relies instead on the intertwining of PDE techniques,
variational arguments and ideas borrowed from dynamical systems via
 Lax--Oleinik formula.

A dynamical and variational point of view of the matter, integrating
the PDE methods, was brought in by the third author (AS) with
collaborators in \cite{SMT14,ISZ}.  This angle allowed detecting the
stochastic character of the problem, displayed by the random
switching nature of the dynamics related to systems.  This approach
led to the definition of an adapted random action functional, which
constitutes a key tool of our analysis as well. Representation
formulae for viscosity (sub)solutions to stationary systems and a
cycle characterization for the Aubry set were derived.\smallskip

A key output of the present paper is to provide another piece of the
dictionary between Weak KAM Theory and weakly coupled systems of
Hamilton--Jacobi equations. The study herein carried on is to be
regarded as a step in the direction of a deeper understanding of the
phenomena taking place at the critical level of the stationary
version of \eqref{intro evolutive system}.  In this line of
research, further issues to be addressed certainly include an
investigation on the dynamical and geometric properties of the Aubry
set and on the differentiability properties of critical subsolutions
on this set, as well as an extension to systems of the theory of
minimizing Mather measures. The variational formulae provided in
this paper seem to be the right tools for this kind of analysis.
\medskip

\subsection*{Presentation of the main results}
All the results can be localized by classical arguments of finite
speed of propagation. Therefore, we have preferred to state them in
the Euclidean space $\R^N$, keeping in mind that they remain valid
on any manifold.

The Lax--Oleinik formula for systems is obtained in the form of
infima of expectations, where the infimum is over a  suitable class
of random curves. The  admissible random curves for this procedure
are defined in Section \ref{sez admissible curves}. The nonlinear
character of the setting makes the structure of the formula more
involved than in the original linear random evolution models. In
this case, in fact, the expectation semigroup is simply obtained via
concatenation on any sample path of the deterministic semigroups
related to the switching operators plus averaging. The nonlinearity
brings in a sense a commutation problem between infimum and
expectation.  In this framework, we perform a key step in the
analysis, notably in view of studying viscosity solutions of the
evolutive system,  by establishing a differentiation  formula for
Lipschitz--continuous functions on admissible curves, see Theorem
\ref{teo key}.

Due to its random character, the formula is painful to  handle
directly. It is not easy to  show for instance  that  it defines a
semigroup of operators on suitable functional spaces or that the
associated value functions are continuous or even semicontinuous in
$(t,x)$. For this reason, we resort to  a rather indirect approach
putting it in relation to the system  via a sub--optimality
principle and  showing first  that the value function, for a
suitable initial datum,  is a viscosity subsolution to the system in
the discontinuous sense, see Section \ref{LO}. The procedure is not
new, but the vectorial character of the problem and the random
setting add a number of additional difficulties. The implementation
therefore requires some new tools and ideas.

Under mild regularity conditions on the initial datum, we moreover
prove in Section \ref{minimo} the existence of minimizing random
curves, namely  curves realizing the infimum in the Lax--Oleinik
formula. This is somehow surprising since in general the presence of
expectation operators makes such an output quite difficult to
obtain. Our strategy is   composed of two steps. We first untangle
the randomness and tackle the optimization problem on any sample
path,
  obtaining in this way,  in general, multiple deterministic minimizers, and  then
  build the sought random minimizer by performing  a measurable
selection.

We get  in Section \ref{regolare} more information on the regularity
of minimizers assuming Tonelli like conditions on the $H_i$.  Given
any  such  minimal random curve,  we prove,  for almost all fixed
sample path, differentiability in any bounded interval up to a
finite number of points. We derive  differentiability of the
solution of the system on such curves plus existence of an adjoint
random curve. Minimizers and adjoint curves  are governed by a
twisted generalized Hamiltonian dynamics.  As extra consequence, we
recover from the scalar case  regularizing properties of the action
 of the associated  semigroup on bounded Lipschitz--continuous initial data.

To complete the outline of the paper, we further point out that
Section \ref{prelimina} contains preliminary material plus notations
and terminology, Section \ref{sys} collects  some basic facts and
definitions on systems, and  Appendix \ref{appendix PDE} is devoted
to the proofs of some needed results for both systems and time--dependent
equations.

 We would finally  like to stress that in the random part, see Section
 \ref{caso}, we avoid as much as possible technicalities and advanced
 probabilistic notions working on spaces of c\`{a}dl\`{a}g  and continuous paths. Hopefully, it makes the presentation
 palatable for PDE oriented readers.

\bigskip

\begin{section}{Preliminaries}\label{prelimina}

%per avere le formule numerate seguendo le sezioni

\numberwithin{equation}{section}

With the symbols $\N$ and $\R_+$ we will refer to the sets of
positive integer numbers and nonnegative  real numbers,
respectively.  Given $k \in \N$, we denote by $\leb^k$ the Lebesgue
measure in $\R^k$. Given $E \subset \R^k$, we say that a property
holds {\em almost everywhere} ($a.e.$ for short) in  $E$ if it holds
up to a subset of $E$ with vanishing $\leb^k$ measure. We say that
$E$ has {\em full measure} if $\leb^k(\R^k \setminus E)=0$.  We
write $\langle \cdot, \cdot \rangle$ for the scalar product in
$\R^k$. We will denote by $\overline E$ the closure of $E$. We will
denote by $B_r(x)$ and $B_r$ the open Euclidean ball of radius $r>0$
centered at $x\in\R^k$ and $0$, respectively. By the term {\em
curve}, we mean throughout the paper a locally absolutely continuous
curve.

\indent
Let $E$ be a Borel subset $E$ of $\R^k$. Given a measurable function
$g:E\to\R$, we will denote by $\|g\|_{L^\infty(E)}$ the usual $L^\infty$--norm
of $g$. When $\mathbf{g}$ is vector--valued, i.e. $\boldg:E\to\R^d$, we will write
\[
\|\boldg\|_{L^\infty(E)}:=\max_{1\leqslant i\leqslant d}\|g_i\|_{L^\infty(E)}.
\]
The above notation will be mostly used in the case when either
$E=\clTcyl$ or $E=\R^N$. In the latter case, we will often write $\|\boldg\|_\infty$
in place of $\|\boldg\|_{L^\infty(\R^N)}$.

We will denote by $\left(\D{BUC}(\R^N)\right)^m$ the
space of bounded uniformly continuous functions $\uu=(u_1,\dots,u_m)^T$ from $\R^N$ to
$\R^m$ (where the upper--script symbol $T$ stands for the
transpose).
A function $\uu:\R^N\to\R^m$ will be termed Lipschitz continuous if
each of its components is $\kappa$--Lipschitz continuous, for some
$\kappa>0$. Such a constant $\kappa$ will be called a {\em Lipschitz
constant} for $\uu$. The space of all such functions will be denoted
by $\left(\D{Lip}(\R^N)\right)^m$. Analogously, for every fixed $T>0$,
we will denote by $\tvBUC$ and $\tvLip$ the space of bounded uniformly
continuous functions and Lipschitz continuous functions
from $\Tcyl$ to $\R^m$, respectively.
\smallskip\par

We will denote by $\1=(1,\cdots,1)^T$ the vector of $\R^m$ with all
components equal to 1. We consider the following partial relations
between elements  $\mathbf{a},\mathbf{b}\in\R^m$: $\aaa \leqslant
\mathbf{b}$ (respectively, $\aaa <\mathbf{b}$)   if $a_i\leqslant
b_i$ (resp., $<$) for every $i\in\ind$. Given two functions
$\uu,\vv:\R^N\to\R^m$, we will write $\uu \leqslant \vv$ in $\R^N$
(respectively, $<$) to mean that $\uu(x)\leqslant \vv(x)$
\big(resp., $\uu(x)<\vv(x)$\big)  {for every $x\in\R^N$}.

\smallskip

Given $n$ subsets $A_i$ of $\R^k$ and $n$ scalars $\lambda _i$,
$i=1, \cdots,n$, we define
\begin{equation}\label{minko}
    \sum_{i=1}^n \lambda_i \, A_i = \left\{x = \sum_i \lambda_i , y_i \mid
    y_i \in A_i \;\;\hbox{for any $i$} \right\}.
\end{equation}
\medskip

We give some definitions and results of set--valued analysis we will
need in what follows, the material is taken from \cite{CV}.  Let $X$,
$Y$ be Polish spaces, namely complete, separable metric spaces,
endowed with the Borel $\sigma$--algebras $\F_X$, $\F_Y$. We denote
by $Z$ a map from $X$ to the compact (nonempty) subsets of $Y$.
Given $E \subset Y$, we set
\[Z^{-1}(E)= \{x \in X \mid Z(x) \cap E \neq \emptyset\}.\]

\smallskip

\begin{definition}\label{uscmea}
    The set--valued map  $Z$ is said {\em upper semicontinuous} if $Z^{-1}(E)$ is closed
    for any closed subset $E$  of $Y$;  $Z$ is said {\em measurable} if $Z^{-1}(E) \in \F_X$ for any
    closed (or alternatively open) subset $E$ of $Y$.
\end{definition}

The next selection result is a simplified version, adapted to our
needs,  of Theorem III.8 in \cite{CV}.

\begin{teorema}\label{CastaingValadier} If the compact--valued map $Z$ is measurable then
it admits a measurable selection, namely there exists a measurable function $f: X
\to Y$ with $f(x) \in Z(x)$ for any $x$ and $f^{-1}(\F_Y) \subset
\F_X$.
\end{teorema}

\medskip

Given a locally Lipschitz continuous  function $u: \R^k \to \R$ and
$x\in \R^k$ we define  the {\em Clarke generalized gradient}  at $x$
as
\[\partial^C u(x) = \mathrm {co} \{ p \in \R^k \mid p = \lim_n Du(x_n), \, x_n \to x\}\]
where co stands for the convex envelope and the approximating
sequences $x_n$ are made up by differentiability points of $u$.
Recall that the function $u$ is differentiable in a set of  full
$\leb^k$ measure thanks to Rademacher Theorem. We record for later
use

\smallskip

\begin{prop}\label{convexconvex} Given a locally Lipschitz continuous function $u: \R^k \to
\R$, the map\ $x \mapsto \partial^C u(x)$\
is convex compact valued and upper semicontinuous.
\end{prop}

\smallskip

 Even if the
following statement is well known, we provide a proof for reader's
convenience.

\smallskip

\begin{lemma}\label{clarke} Let $u: \R^k \to \R$, $\eta: \R_+  \to \R^k$ be a locally Lipschitz
continuous function and a locally  absolutely continuous curve,
respectively. Let $s \geqslant 0$ be such that $t \mapsto u\big(\eta(t)\big)$ and
$t \mapsto \eta (t)$ are both differentiable at $s$. Then
\[ \frac \dd{\dd t} u\big(\eta(t)\big)_{\mbox{\Large $|$}_{t=s}} = \langle p,
\dot \eta(s) \rangle \quad\hbox{for some
 $p \in \partial^C u\big(\eta(s)\big)$.}\]
\end{lemma}
\begin{proof} The function $u \circ \eta$ is clearly locally absolutely continuous. We start  from the relation
\begin{equation}\label{liplip2}
    \limsup_{y \to x \atop h \to 0^+} \frac {u(y+h \, q)- u(y)}h =
\max \{\langle p, q \rangle \mid  p \in \partial^C u(x)\}
\end{equation}
which holds true for any $x$, $q$ in $\R^k$, see \cite[ pp. 195--196 and 208]{Cl13}. If $s$
satisfies the assumptions then
\[ \frac \dd{\dd t} u\big(\eta(t)\big)_{\mbox{\Large $|$}_{t=s}} = \lim_{h \to 0} \frac
{u\big(\eta(s+h)\big)- u\big(\eta(s)\big)}h\] and
\[ \lim_{h \to 0} \frac {u\big(\eta(s+h)\big)- u\big(\eta(s)+ h \, \dot\eta(s)\big)}h =0\]
which implies
\[ \frac \dd{\dd t} u\big(\eta(t)\big)_{\mbox{\Large $|$}_{t=s}} =
\lim_{h \to 0^+} \frac {u\big(\eta(s)+h \, \dot\eta(s)\big)- u\big(\eta(s)\big)}h\]
and taking into account \eqref{liplip2} we get
\begin{equation}\label{liplip3}
    \frac \dd{\dd t} u\big(\eta(t)\big)_{\mbox{\Large $|$}_{t=s}}
     \leqslant \max \{\langle p, \dot\eta(s) \rangle \mid p \in \partial^C
u(x)\}.
\end{equation}
 We further have
 \[ -\frac \dd{\dd t} u\big(\eta(t)\big)_{\mbox{\Large $|$}_{t=s}} =
 \lim_{h \to 0^+} \frac {u\big(\eta(s)+h \,\big(- \dot\eta(s)\big)\big)- u\big(\eta(s)\big)}h\]
and consequently
\begin{eqnarray}
  \frac \dd{\dd t} u\big(\eta(t)\big)_{\mbox{\Large $|$}_{t=s}}  &\geqslant&
  - \max \big\{\big\langle p, \big(- \dot\eta(s)\big) \big\rangle \mid p \in
    \partial^C u(x)\big\} \label{liplip4} \\
   &=& \min \big\{\big\langle p, \dot\eta(s) \big\rangle \mid p \in \partial^C
u(x)\big\} \nonumber
\end{eqnarray}
Bearing in mind that $\partial^C u\big(\eta(s)\big)$ is convex, we directly
deduce the assertion  from \eqref{liplip3} and  \eqref{liplip4}.
\end{proof}

\medskip

We write down, in view of future use, a version of
Denjoy--Young--Saks Theorem, see \cite[pp. 17--19]{RN}.  A
definition is preliminarily needed:  for a real valued function $f$,
the upper right and lower right Dini derivative at a point $s$ are
given, respectively, by
\[
\limsup_{h \to 0^+} \frac{f(s+h)-f(s)}h,
\qquad\qquad
\liminf_{h \to 0^+} \frac{f(s+h)-f(s)}h.
\]

\smallskip

\begin{teorema} Let $f$ be a real valued function defined on an
interval. Then outside a set of  vanishing  $\leb^1$ measure the
following condition holds true: if $f$ is not differentiable  at $s$
then one of the two right Dini derivatives must be infinite.
\end{teorema}

\smallskip

As at a point $s$ where $f$ admits a right derivative, both right Dini derivatives are finite (and equal),
an immediate corollary is:

\smallskip

\begin{cor}\label{DYS} Let $f$ be a real valued function defined on an
interval. If $f$ is right differentiable a.e. then it is
differentiable a.e.
\end{cor}

% \medskip
%
% We consider the sup--convolution restricted to a bounded domain of
% $\R_+ \times \R^N$. Given a vector valued function $\uu$ the
% sup--convolution must be understood componentwise.
%
%
% \smallskip
%
%
% \begin{lemma}\label{lemlem}  Assume $u$ to be continuous and fix a bounded domain
% $B$, then
% \[ \frac{|x-y|^2}\delta = \OO(\eps) \qquad\hbox{for  any $x \in B$, $y$
% $u^\delta$--optimal for $x$.}\]
% \end{lemma}
%
%
%
% \begin{proof}
%  Let $x$, $y$ be as in the statement, we denote by $\nu$ an
%  uniform continuity modulus of $u$ in $B$ . We know that
% \begin{equation}\label{lemlem1}
%     |x-y|  \leqslant   2 \, \sqrt{R \,\delta},
% \end{equation}
%  where $R = \sup_B |u|$,   then taking into account
%  \eqref{lemlem1} we get
%  \[ \frac{|x-y|^2}{2 \, \delta} \leqslant u(y)- u(x) \leqslant \omega(|x-y|)
%   \leqslant \omega (2 \, \sqrt{R \,\delta}).\]
%
% \end{proof}

\medskip

\end{section}

\bigskip

\begin{section}{Weakly coupled systems}\label{sys}

We consider  the evolutionary weakly coupled system
\begin{equation}\label{HJS}\tag{HJS}
\partial_t u_i+H_i(x,D_xu_i)+\big(B\uu(t,x)\big)_i=0\quad\hbox{in
$(0,+ \infty)\times\R^N$}\quad\hbox{$ i\in\ind$,}
\end{equation}
where  we have denoted by
$\uu(t,x)=\big(u_1(t,x),\dots,u_m(t,x)\big)^T$ the vector--valued
unknown function. We assume the Hamiltonians $H_i$ to satisfy, for $
i \in \ind$
\smallskip

\begin{itemize}
    \item[(H1)] \quad $H_i \in\D{UC}(\R^N\times B_R)$ \quad for every $R>0$;\smallskip\\
    \item[(H2)] \quad $p\mapsto H_i(x,p)\qquad\hbox{is  convex on $\R^N$ for any
    $x\in \R^N$;}$ \smallskip\\
     \item[(H3)] \quad there exist two superlinear functions $\alpha,\beta:\R_+\to\R$ such that
     \[
     \qquad \alpha(|p|)\leqslant H_i(x,p)\leqslant \beta(|p|)\quad \hbox{for every $(x,p)\in\R^N\times\R^N$}.
     \]
    \end{itemize}
By {\em superlinear} we mean that
$$\lim\limits_{h\to +\infty}\frac{\alpha(h)}{h}=\lim\limits_{h\to +\infty}\frac{\beta(h)}{h}=+\infty.$$
It is easily seen that the continuity modulus of $H_i$ in $\R^N\times B_R$ and the functions $\alpha,\,\beta$ can be chosen independently of $i\in\ind$.

We  define the {\em Fenchel transform} $L_i:\R^N\times\R^N\to \R$ of
$H_i$ via
\begin{equation}\label{def L}
L_i(x,q):=\sup_{p\in\R^N}\left\{\langle p, q\rangle - H_i(x,p)\right\}.
\end{equation}
The function $L_i$ is called the Lagrangian associated with the
Hamiltonian $H_i$ and satisfies properties analogous to (H1)--(H3).

The {\em coupling matrix} $B=(b_{ij})$ has dimensions  $m\times m$
and satisfies
\begin{itemize}
 \item[(B)] \quad $b_{ij}\leqslant 0\quad \hbox{for $j\not=i$,}\qquad \sum\limits_{j=1}^m b_{ij}=0
\quad \qquad\hbox{for every $i\in\{1,\dots,m\}$.}$
\end{itemize}
We will denote by $\big(B\uu(t,x)\big)_i$  the $i$--th component of
the vector $B\uu(t,x)$, i.e.
\[
\big(B\uu(t,x)\big)_i=\sum\limits_{j=1}^m b_{ij}u_j(t,x).
\]

\medskip

\begin{oss}\label{oss monotone system}
The weakly coupled system \eqref{HJS} is a particular type of {\em
monotone system}, i.e. a system of the form $
 G_i\big(t,x,u_1(x),\dots,u_m(x),D_xu_i\big)=0$ in $\R^N$ for every $i\in\ind$,
where suitable monotonicity conditions  with respect to the
$u_j$--variables are assumed on the functions $G_i$, see
\cite{CamLey,Engler,Is92,IsKo,Len88}. In the  case  under
investigation, the conditions assumed on the
 coupling matrix imply, in particular, that each function $G_i$ is non--decreasing in $u_i$, and non--increasing in $u_j$ for
  $j\not=i$, for every $i\in\ind$.
\end{oss}

Given a  function $u$ on $\cyl$, we will call  {\em subtangent}
(respectively, {\em supertangent}) of $u$ at $(t_0,x_0)\in\cyl$ a
function $\phi$ of class $ \D{C}^1$ in a neighborhood  of $(t_0,x_{0})$
such that $u-\phi$ has a local minimum (resp., maximum) at $x_{0}$.
The  differentials of subtangents (resp. supertangents)
$\big(\partial_t \phi(t_0,x_0),D_x\phi(t_0,x_0)\big)$ make up  the
{\em subdifferential}  (resp. {\em superdifferential}) of $u$ at
$(t_0,x_0)$,   denoted $D^-u(t_0,x_0)$ \big(resp. $D^+u(t_0,x_0)$\big).  The
function $\phi$ will be furthermore termed {\em strict subtangent}
(resp., {\em strict supertangent}) if $u-\phi$ has a {\em strict}
local minimum (resp., maximum) at $(t_0,x_{0})$.

Given a function $\uu: \R_+ \times \R^N \to \R^m$ locally bounded
from above (resp. from below), we define its {\em upper
semicontinuous envelope} $\uu^*$ (resp. {\em lower semicontinuous
envelope} $\uu_*$)  as follows:
\begin{eqnarray*}
   u_i^*(t,x):=\limsup_{(s,y)\to (t,x)} u_i(s,y)\qquad\hbox{for every $(t,x)\in\clcyl$ and $i\in\ind$,}
\end{eqnarray*}
(resp. ${u_i}_*(t,x):=\liminf\limits_{(s,y)\to (t,x)} u_i(s,y)$ for
every $(t,x)\in\clcyl$ and $i\in\ind$).

\begin{definition}\label{defvisco}
 We will say that $\uu:\cyl \to\R^m$ locally bounded from above
 is a {\em viscosity subsolution} of \eqref{HJS} if
\begin{equation*}
p_t+H_i(x,p_x)+\big(B\uu^*(t,x)\big)_i\leqslant 0
\end{equation*}
for every $(t,x,i)\in\cyl\times\ind$, $(p_t,p_x)\in D^+ u_i^*(t,x)$.

We will say that $\uu:\cyl \to\R^m$ locally bounded from below
 is a {\em viscosity supersolution} of \eqref{HJS} if
\begin{equation*}
p_t+H_i(x,p_x)+\big(B\uu_*(t,x)\big)_i\geqslant 0
\end{equation*}
for every $(t,x,i)\in\cyl\times\ind$, $(p_t,p_x)\in D^-{
u_i}_*(t,x)$.

We will say that a locally bounded function  $\uu$ is a {\em
viscosity solution} if it is both a sub and a supersolution.
\end{definition}
In the sequel, solutions, subsolutions and supersolutions will be
always meant in the viscosity sense, hence the adjective {\em
viscosity} will be  omitted.
\medskip

Due to the continuity and convexity  properties of the Hamiltonians
$H_i$, we have:

\begin{prop}\label{prop subsol equivalent def}
Let $\uu$ be locally Lipschitz in $(0, + \infty) \times \R^N$. The
following properties are equivalent
\begin{itemize}
    \item[ (i)]$\uu$ is a (viscosity) subsolution of \eqref{HJS};
    \item[(ii)] $\uu$ is an  almost everywhere subsolution, i.e. for any $i\in\ind$
    \begin{eqnarray*}
 \partial_t u_i(t,x)+H_i\big(x,D_xu_i(t,x)\big)+\big(B\uu(t,x)\big)_i\leqslant 0\quad\hbox{for a.e. $(t,x)\in\cyl$;}
 \end{eqnarray*}
\item[(iii)] $\uu$ is a  Clarke subsolution, i.e.
\begin{eqnarray*}
 r +H_i\big(x,p\big)+\big(B\uu(t,x)\big)_i\leqslant 0\qquad\hbox{for every  $(r,p) \in \partial^C u_i(t,x)$,}
 \end{eqnarray*}
for every $(t,x,i)\in\cyl\times\ind$.
\end{itemize}
\end{prop}

\vspace{2ex}

The following holds:

\begin{teorema}\label{teo existence evo wcs}
Let $\uu^0\in \big(\D{BUC}(\R^N)\big)^m$. There exists a unique
solution $\uu(t,x)$ of  \eqref{HJS} in  $(0,+ \infty) \times \R^N$
agreeing with $\uu^0$ at $t =0$,  which belongs to
$\big(\D{BUC}([0,T]\times\R^N)\big)^m$ for any  $T > 0$. If $\uu^0$
is furthermore assumed Lipschitz continuous,
 then, for every $T>0$, $\uu\in\tvLip$,  with  Lipschitz
constant  solely  depending on  $H_1,\dots,H_m$ and $B$, on
$\|\uu^0\|_{L^\infty(\R^N)}$,
$\|D\uu^0\|_{L^\infty(\R^N)}$, and on $T$.\medskip
\end{teorema}

The uniqueness of the solutions provided by the previous theorem is in fact a consequence of the following comparison principle:

\begin{prop}\label{prop comparison}
Let  $\vv,\,\ww:[0,+ \infty )\times\R^N\to\R^m$ be an upper
semicontinuous subsolution and a lower semicontinuous
supersolution of \eqref{HJS}, respectively. Let us assume that $\vv$ and $\ww$ are bounded in $\clTcyl$, for every $T>0$, and that either
$\vv(0,\cdot)$ or $\ww(0,\cdot)$ are in $\D{BUC}(\R^N)$.
Then
\[
 v_i(t,x)-w_i(t,x)\leqslant\max_{1\leqslant i \leqslant m}\,\sup_{\R^N} \big(v_i(0,\cdot)-w_i(0,\cdot)\big)
\]
for all $(t,x)\in\ccyl$ and $i\in\ind$.
\end{prop}

The above stated results are, essentially, a consequence of what is proved in \cite{CamLey}, see Appendix \ref{appendix PDE} for more details.

\end{section}

\begin{section}{Random frame}\label{caso}

\begin{subsection}{Definitions and terminology}\label{sez prob}
In this subsection we make precise the random frame in which our
analysis takes place.  This is basically a way of introducing the
Markov chain generated by $- B$.

Following the constructive approach of  \cite{SMT14},  we take as
sample space the space of paths
\[ \omega:\R_+\to\ind \]
 that are right--continuous and possess
left--hand limits, denoted by $\Omega$.  These are known in
literature as {\em c\`{a}dl\`{a}g paths}, a French acronym for {\em
continu \`a droite, limite \`a gauche}.

We refer the reader to the magnificent book of Billingsley
\cite{Bill99} for a detailed treatment of the topics. By
c\`{a}dl\`{a}g property  and the fact that  the range of
$\omega\in\Omega$ is finite, the points of discontinuity of  any
such path are isolated and consequently finite in compact intervals
of $\R_+$ and countable (possibly finite) in the whole of $\R_+$. We
call them {\em jump times} of $\omega$.

The space  $\Omega$ is endowed  with a distance, named after  {\em
Skorohod}, see \cite{Bill99},  which turns it into a Polish space.
We denote by $\F$ the corresponding Borel  $\sigma$--algebra and,
for every $t\geqslant 0$, by $\pi_t:\Omega\to\ind$ the map that
evaluates each $\omega$ at $t$, i.e. $\pi_t(\omega)=\omega(t)$ for
every $\omega\in\Omega$.
It is known that $\F$ is the minimal $\sigma$--algebra that makes
all the functions $\pi_t$ measurable, i.e.
 $\pi_t^{-1}(i)\in\F$ for every $i\in\ind$ and $t\geqslant 0$. In
 other terms, the family $\CC$ of {\em cylinders}
\[\mathcal C(t_1, \cdots, t_k;i_1, \cdots, i_k)=
\left\{\omega\in\Omega\
\mid\omega(t_1)=i_1,\dots,\omega(t_k)=i_k\right\},\] with
$0\leqslant t_1<t_2<\dots<t_k,\ i_1,\dots,i_k\in\ind$ and $k\in\N$,
generates  $\F$.  Given  $t \geqslant 0$, the $\sigma$--algebra
generated by the cylinders of $\CC$  enjoying  the additional
property that $t_k \leqslant t$, is denoted by $\F_t$. Then
$\left\{\F_t\right\}_{t\geqslant 0}$ is a {\em filtration of $\F$},
 i.e. $\F_s\subseteq \F_t$ for every $0\leqslant s <t$ and $\cup_{t\geqslant 0}\,\F_t=\F$.
 It is in addition {\em right--continuous} in the sense that  $\F_t= \cap_{s >t} \F_s$ for any $t$.
 Note that $\F_0$ comprises a finite number of sets, namely $\Omega$,
$\varnothing$, $\Omega_i:=\{\omega\in\Omega\,:\,\omega(0)=i\,\}$
for every $i \in \ind$, and unions of such sets.
The cylinders constitute  a  {\em separating class}, in the sense
that any probability measure on $\F$  is identified by the values
taken on $\CC$,  see Theorem 16.6 in \cite{Bill99}.

Let $\mu$ be a probability measure on $(\Omega,\F)$. Given $E \in \F$, we define the {\em restriction } of $\mu$ to $E$
as\ \ $\mu \restr E (F) = \mu (E \cap F) \quad\hbox{for any $F \in
\F$}.$
The probability $\mu$ conditioned to the event $E\in\F$ is defined as
\begin{eqnarray*}
\mu(F\,|\,E):=\frac{\mu(F\cap E)}{\mu(E)}\qquad\hbox{for every $F\in\F$},
\end{eqnarray*}
where we agree that $\mu(F\,|\,E)=0$ whenever $\mu(E)=0$.

Let us now fix an $m\times m$ matrix $B$ satisfying assumption
(B). We record that $\e^{-tB}$ is a stochastic matrix for every $t\geqslant 0$, namely
a matrix with nonnegative entries and with each row summing to 1, see for instance
Appendix A in \cite{SMT14}. We endow $\Omega$ of a probability measure $\PP$ defined on
the $\sigma$--algebra $\F$ in such a way that the right--continuous
process $\left(\pi_t\right)_{t\geqslant 0}$ is a {\em Markov chain
with generator matrix $-B$}, i.e. it satisfies the Markov property
\begin{eqnarray}\label{Markov property}
\PP\big(\omega(t_k)=i_k\,|\,\omega(t_1)=i_1,\dots,\omega(t_{k-1})=i_{k-1}\,\big)
= \left(\e^{-B(t_k-t_{k-1})}\right)_{i_{k-1}i_k}
\end{eqnarray}
for all times $0\leqslant t_1<t_2<\dots<t_k$, states
$i_1,\dots,i_k\in\ind$ and $k\in\N$. For the existence and an explicit construction of such
a probability measure, we refer the reader to \cite{SMT14}. We will denote by $\PP_i$ the
probability measure $\PP$ conditioned to the event $\Omega_i$ and write $\EE_i$ for the
corresponding expectation operators. These entities will constitute
the basic building blocks of our analysis. It is easily seen that the Markov property
\eqref{Markov property} holds with $\PP_i$ in place of $\PP$, for every $i\in\ind$.

%
% A converse construction is also possible. Given  $ i \in \ind$, and the
% corresponding vector $\mathbf e_i$ of the canonical basis of  $
% \R^m$,  we associate to any  $\mathcal C(t_1, \cdots, t_k;i_1,
% \cdots, i_k)$ the quantity
% \[\left (\mathbf e_i \,\e^{-t_1 B} \right )_{i_1} \, \prod_{j=2}^{k}
% \left (\e^{-(t_l-t_{l-1})B} \right )_{i_{j-1} \,i_j}.\] By Kolmogorov
% Extension theorem, there exists an unique probability measure,
% denoted by $\PP_i$,  which extends the previous function to the whole
% of $\F$, see  \cite[Theorem 14.36]{K1}.

We proceed by introducing some more notations and terminology. We  call
{\em random variable} a  map $X:(\Omega,
\F)\to\big(\FF,\Bor(\FF)\big)$, where $\FF$ is a Polish space and
$\Bor(\FF)$  its Borel $\sigma$--algebra, satisfying
$X^{-1}(A)\in\F$ for every $A\in\Bor(\FF)$.

Given a probability measure $\mu$  on $(\Omega,\F)$,  we denote by
$X_\#\mu$ the {\em push--forward} of $\mu$ through the map $X$, i.e.
the probability measure on $\Bor(\FF)$ defined as
\[\left({X}_\#\mu\right)(A):=\mu\left(\{\omega\in\Omega\,:\,X(\omega)\in A\,\}\right)\quad\hbox{for
every $A \in \Bor(\FF)$.}\]

A probability measure $\nu$ on $\I:=\ind$ will be identified with a {\em probability vector} $\aaa\in\R^m$, i.e. a vector
with nonnegative components summing to 1, via the formula \
$\aaa\cdot\bbb:=\int_{\I} \bbb\,\dd \nu(i)$\ \ for every $\bbb\in\R^m$.

%
% \smallskip
%
% \begin{definition}\label{stoppa} A random variable $\tau$ taking nonnegative
% values is said a {\em stopping time} if
% \[\{\omega \mid \tau(\omega) < t\} \in \F_t \qquad\hbox{for every $t \geqslant 0$.}\]
% This immediately implies that $\{\tau \leqslant t \}$ and $\{\tau =
% t\}$ belong to $\F_t$.
% \end{definition}
%

\smallskip

\subsection{Basic facts on $\PP_i$ and $\EE_i$}

 In this subsection we gather for later use some properties of
 probability measures $\PP_i$ and the corresponding expectation
 operators.
 \smallskip

\begin{lemma}\label{differ}  There exists a constant $C$ such
 that
 \[\PP_i \big (\{\omega \mid \omega(t_1) \neq \omega(t_2)\} \big ) \leqslant C \,
 |t_1-t_2| \qquad\hbox{for any $i \in \ind$, $t_1$, $t_2$ in $\R_+$}.\]
\end{lemma}
\begin{proof}  We fix  an index $i$ and  assume $t_2 >t_1$. We denote by $F$ the  set in
the statement. Then \ $F = \bigcup_{j=1}^m \bigcup_{k \neq j} \mathcal C(t_1,t_2; k,j)$\
and
\[
\PP_i \big ( \mathcal C(t_1,t_2; k,j) \big )
=
\PP_i \big(\omega(t_2)=j \mid \omega(t_1)=k \big)\, \PP_i \big( \omega(t_1)=k \big).
\]
By making use of the Markov property \eqref{Markov property} for $\PP_i$ and of the fact that $\e^{-tB}$ is a stochastic matrix we infer
\[
\PP_i(F)
\leqslant
\sum_{\substack{j, k=1 \\ j\not=k}}^m (\e^{- t_1 B})_{ik} \, (\e^{- (t_2-t_1) B})_{k j}
\leqslant
\sum_{\substack{j, k=1 \\ j\not=k}}^m (\e^{- (t_2-t_1) B})_{k j}.
\]
The assertion is obtained by taking into account that the rightmost
term in the above inequality is 0 when $t_2=t_1$ and all the entries
of the matrix $\e^{-tB}$ are Lipschitz continuous functions in
$\R_+$.
\end{proof}

\medskip

Given  $t \geqslant 0$ and an index $i$, the components of
${\pi_t}_{\#}\PP_i$  are equal to  $\PP_i \big(\mathcal C(t;j)\big)$, for
any $i$, and we  deduce from the definition of $\PP_i$
\[{\pi_t}_{\#}\PP_i  = \mathbf e_i \, \e^{-Bt } .\]
 This implies
for $0<s<t$
\begin{equation}\label{marziotta}
 \EE_i v_{\omega(t)}=  \sum_j v_j \, {\pi_t}_{\#}\PP_i (j) = \left(\e^{-Bt}\vv\right)_{i}=
 \sum_j (\e^{-B(t-s)}\vv)_j \, {\pi_s}_{\#}\PP_i (j)
\end{equation}
for any $\vv \in \R^m$.  We aim to extend the above formula with
a  random variable taking values in $\R^m$ in place of a constant
vector.  The task will be performed via approximation by simple
random variables, a difficulty is that while a constant vector is
trivially  $\F_0$--measurable, a general $\R^m$--valued random
variable is related in a more involved way to the filtration $\F_t$.
As a preliminary step, we recall from \cite[Lemma 3.4]{SMT14}.

\smallskip

\begin{lemma}\label{lemma index derivative}
Let $s\geqslant 0$  and $E\in\F_s$. Then for every $i\in\ind$ and $t
\geqslant s$
\[{\pi_t}_{\#}\left(\PP_i\restr E\right)=\left({\pi_s}_{\#}
 \left(\PP_i\restr E\right)\right)\e^{-B(t-s)}.\]
\end{lemma}

\smallskip

\begin{prop}\label{preprop index derivative}
Let $s\geqslant 0$ and $i\in\ind$. Let $\mathbf g=(g_1, \cdots,
g_m)$ be an $\F_s$--measurable random variable taking values in
$\R^m$, which is, in addition, bounded in $\Omega_i$. Then
\begin{eqnarray}\label{claim1 prop index derivative}
 \EE_i \left [g_{\omega(t)}(\omega)\right ] = \EE_i\Big [  \big(\e^{-B(t-s)}\mathbf g(\omega)\big)_{\omega(s)}\Big]
 \qquad\hbox{for every  $t\geqslant s$.}
\end{eqnarray}
\end{prop}
\vspace{1ex}

\begin{proof} We first assume $\mathbf g$ to be simple, namely\ $\mathbf g =\sum_{k=1}^{l} \xi^k\,\cchi_{E_k}$\
for some $l \in \N$,
  vectors $\xi^k\in\R^m$ and $\F_s$--measurable sets
  $E_k\subset\Omega$.
By exploiting Lemma \ref{lemma index derivative}, we get
\begin{eqnarray*}
  \EE_i\left[ g_{\omega(t)}(\omega) \right]&=& \sum_k  \EE_i  [\xi^k_{\omega(t)} \,
  \cchi_{E_k}] = \sum_k \, \sum_j \xi^k_j \, (\PP_i \restr E_k)\big( \mathcal C(t;j)\big) \\
   &=& \sum_k  \left ({\pi_t}_{\#}\left(\PP_i\restr E_k \right) \right  ) \cdot \xi^k =
   \sum_k  \left ( {\pi_s}_{\#}\left(\PP_i \restr E_k\right) \, \e^{-B
   (t-s)} \right ) \cdot
   \xi^k\\ &=& \sum_k  \, \sum_j \left (\e^{-B
   (t-s)} \, \xi^k \right )_j   \, (\PP_i \restr E_k)\big( \mathcal C(s;j)\big) =  \sum_k
   \EE_i  \left [\left (\e^{-B
   (t-s)} \,
   \xi^k \right )_{\omega(s)}\, \cchi_{E_k} \right ] \\
   &= &\EE_i \left(\e^{-B(t-s)} \mathbf
   g(\omega)\right)_{\omega(s)}.
\end{eqnarray*}
This shows   the assertion  for simple random variables. For a
general $\mathbf g$, there exists, see \cite[Theorem 1.4.4, Chapter
1]{Kut13}, a sequence of $\R^m$--valued $\F_s$--measurable random
variables $\mathbf g_n$ with
\ \ $\mathbf g_n(\omega) \to \mathbf g(\omega) \ \ \hbox{for all $\omega \in
\Omega$}$\  and $\mathbf g_n$ bounded in $\Omega_i$.
 Since \eqref{claim1 prop index derivative} holds true for
$\mathbf g_n$ thanks to the first part of the proof,  we pass to the
limit on both side of the formula \eqref{claim1 prop index
derivative} exploiting the boundedness of $\mathbf g$ on $\Omega_i$
and using the Dominated  Convergence Theorem. This ends the proof.
\end{proof}

\medskip

Differentiating under the integral sign, we derive from
\eqref{claim1 prop index derivative}:

\smallskip

\begin{prop}\label{prop index derivative}
Let $s\geqslant 0$ and $i\in\ind$. Let $\mathbf g$  be an
$\F_s$--measurable random variable taking values in $\R^m$,
bounded in $\Omega_i$. Then the function \  $t \mapsto \EE_i\left[
g_{\omega(t)}(\omega)\right]$\ is differentiable in $(s,+\infty)$
and right--differentiable at $s$. Moreover
\begin{eqnarray*}
 \frac{\dd}{\dd t}\EE_i\left[ g_{\omega(t)}(\omega)\right]
 =
-\EE_i\left[\left( \e^{-B (t-s)} \, B \mathbf g
(\omega)\right)_{\omega(s)}\right]\quad\hbox{for every $t\geqslant s$},
\end{eqnarray*}
where the above formula must be understood in the sense of right
differentiability at $t=s$.
\end{prop}

\smallskip

It is worth pointing out that what  matters most in the later
application of the above result is   actually the right
differentiability of the expectations at the initial time $s$.

\end{subsection}
\end{section}

\begin{section}{Admissible curves}\label{sez admissible curves}

\begin{subsection}{Definition and basic properties}

A major role in our construction will be played by the notion of
{\em admissible curve}. In the sequel and throughout the paper, we will denote by
$\D C \left (\R_+;\R^N\right )$ the Polish space of continuous paths taking values in $\R^N$, endowed with a metric that
induces the topology of local uniform convergence in $\R_+$.

\begin{definition}\label{def admissible curve}
We call  admissible curve  a random variable $\gamma:\Omega\to\D
C\left(\R_+;\R^N\right)$ such that
\begin{itemize}
 \item[\em (i)] it  is uniformly (in $\omega$) locally  (in $t$) absolutely continuous,
 i.e. given any bounded interval $I$ and $\eps >0$,  there is $\delta_\eps >0$
 such that
 \begin{equation}\label{h1 control}
  \sum_j (b_j-a_j) < \delta_\eps  \; \Rightarrow  \; \sum_j |\gamma(b_j,\omega)- \gamma(a_j,\omega)|
  < \eps
   \end{equation}
   for any finite family  $\{(a_j,b_j)\}$ of pairwise disjoint intervals
   contained in I and for any $\omega \in \Omega$;
 \item[\em (ii)]  it is {\em nonanticipating}, i.e. for any $t \geqslant 0$
 \begin{equation}\label{h2 control}
  \omega_1\equiv\omega_2\ \hbox{in $[0,t]$}\quad\Rightarrow \gamma(\cdot,\omega_1)\equiv \gamma(\cdot,\omega_2)\ \hbox{in $[0,t]$}.\medskip
 \end{equation}
\end{itemize}
The latter condition, with $t=0$, implies that for any admissible
curve $\gamma$,
  $\gamma(0,\omega)$  is constant on $\Omega_i$, $i \in \ind$.  We refer to this value as
  the  {\em  starting point} of $\gamma$ on $\Omega_i$.

 We will say that $\gamma$ is an admissible curve starting at $x\in\R^N$ when $\gamma(0,\omega)=x$ for every $\omega\in\Omega$.
\end{definition}

\smallskip

\begin{oss} \label{oss admissible curve} Item (i) in Definition \ref{def admissible curve} is equivalent
to one of the following statements, see  \cite[Theorem 2.12]{BGH}:
\begin{itemize}
 \item[(a)] the derivatives $\dot \gamma(t,\omega)$ have locally
equi--absolutely continuous integrals, i.e. for any bounded
interval $I$ in $\R_+$ and $\eps >0$ there is $\delta>0$ such that
\[
\sup_{\omega\in\Omega}\int_J |\dot \gamma(t,\omega)| \, \dd t < \eps \qquad\hbox{for any $J \subset I$ with $|J| < \delta$;}
\]
\item[(b)] there exists a superlinear function $\Theta:\R_+\to\R$ (that can be taken convex and increasing as well) such that, for every bounded interval $I$ in $\R_+$,
\[
 \sup_{\omega\in\Omega}\int_I \Theta\big(|\dot \gamma(t,\omega)|\big) \, \dd t <+\infty.
\]
\end{itemize}
This in
particular implies that lengths of the curves $t \mapsto
\gamma(t,\omega)$ in $I$ are equi--bounded with respect to $\omega$.
Item (ii) will be crucial in the subsequent analysis and can be
equivalently rephrased by requiring that $\gamma(t,\cdot)$ is {\em
adapted}, for any $t$, to the filtration $\F_t$, meaning that
\[\gamma(t,\cdot): \Omega \to \R^N \quad\hbox{is $\F_t$--measurable
for any $t$.}\] Being the paths $s\mapsto \gamma(s,\omega)$
continuous, this is in turn equivalent to a joint--measurability
condition that will be essentially exploited in what follows. More
precisely, $\gamma$ is {\em progressively measurable}, in the sense
that  for any $t \geqslant 0$ the map
\begin{equation}\label{joint}
    \gamma:[0,t]\times\Omega\to\R^N \quad\hbox{is $\Bor([0,t])\otimes\F_t$--measurable.}\\
\end{equation}
\end{oss}

\medskip

It is understood  that in all the previous measurability conditions
$\R^N$ is equipped  with the Borel  $\sigma$--algebra corresponding
to the natural topology.

It is clear that the admissible curves make up
a vector space  with the natural sum and product by a scalar. We
 proceed by establishing some differentiability properties for this
 kind of curves.

% For any control $v$ and any $x\in\R^N$, we define a random variable $\gamma_{x,v}:\Omega\to\D{C}(\R_+;\R^N)$ by setting
% \begin{equation}\label{def curva}
%  \gamma_{x,v}(t)=x+\D{proj}\left(\int_0^t v(s,\omega(s))\,\dd s\right)\qquad\hbox{for any $t>0$,}
% \end{equation}
% %
% where $\D{proj}$ denotes the projection from $\R^N$ to $\R^N$.

\begin{lemma}\label{lem diff points}
For any admissible curve  $\gamma$  the set
\[ \left\{(t,\omega)\in \R_+\times\Omega\,:\,\gamma(\cdot,\omega)\ \hbox{is not
differentiable at $t$}\,\right\} \] belongs to the product
$\sigma$--algebra $\Bor(\R_+)\otimes\F$ and  has vanishing $\leb^1
\times \PP$ measure.
\end{lemma}
\begin{proof} Since measurability properties of a vector valued map and those of its
components are equivalent, we can assume without loosing generality
that $m=1$. We  set for any $(t,\omega)$
\begin{eqnarray*}
  D^+ \gamma(t,\omega) &=& \limsup_{s \to 0} \frac{\gamma(t+s,\omega)-\gamma(t,\omega)}s =
  \lim_{k \to + \infty} \sup_{|s|< 1/k \atop s \in \Q \setminus \{0\}}
  \frac{\gamma(t+s,\omega)-\gamma(t,\omega)}s \\
  D^- \gamma(t,\omega) &=& \liminf_{s \to 0} \frac{\gamma(t+s,\omega)-\gamma(t,\omega)}s =
  \lim_{k \to + \infty} \inf_{|s|< 1/k \atop s \in \Q \setminus \{0\}}
  \frac{\gamma(t+s,\omega)-\gamma(t,\omega)}s.
\end{eqnarray*}
From the above formulae and the fact that the admissible curves make
up  a vector space, we derive that both $D^+ \gamma(t,\omega)$, $D^-
\gamma(t,\omega)$ are $\Bor(\R_+)\otimes\F$ measurable, then the set
in the statement belongs to $\Bor(\R_+)\otimes\F$ as well, since it
can be expressed as
\[ \{(t,\omega)\mid D^+ \gamma(t,\omega) > D^- \gamma(t,\omega)\}.\]
Moreover   its $\omega$--sections have zero Lebesgue measure, being
$\gamma(\cdot,\omega)$  an absolutely continuous  curve for a.e.
$\omega\in\Omega$. This implies that it has vanishing $\leb^1 \times
\PP$ measure, as it was claimed.
\end{proof}

\medskip

Thanks to the previous result, the map
\[
\dot\gamma: [0, + \infty) \times \Omega \to \R^N
\]
associating to any $(t,\omega)$ the derivative of
$\gamma(\cdot,\omega)$ at $t$ is well defined, up to giving an
arbitrary value on the $\leb^1 \times \PP$--null set where the
derivative does not exist. Such a map is progressively measurable, as it is clarified by the next
\smallskip

\begin{prop}\label{dotga} The map $\dot\gamma$ is $\Bor([0,t])\otimes\F_t$--measurable.
\end{prop}
\begin{proof} Looking at the definition of $D^+\gamma$, $D^- \gamma$
provided in the proof of Lemma \ref{lem diff points}, and taking
into account \eqref{joint} and that the filtration $\{\F_t\}$ is
right--continuous, we see that both $D^+\gamma$, $D^- \gamma$ are
$\Bor([0,t])\otimes\F_t$ progressively measurable and that for any
fixed $t$
\[ \dot \gamma= D^+\gamma = D^- \gamma  \qquad\hbox{in $[0,t]
\times \Omega$}\] up to a set of vanishing $\leb^1 \times \PP$ measure
belonging to $\Bor([0,t])\otimes\F_t$. This gives the assertion.
\end{proof}

\smallskip

\begin{cor}\label{leb} For any $i \in \ind$, a.e. $s \in \R_+$ we have
\begin{equation}\label{leb1}
   \lim_{h \to 0^+} \EE_i \left [ \frac 1h \, \int_s^{s+h}
|\dot\gamma(t,\omega)| \, \dd t \right ] = \EE_i\Big[
|\dot\gamma(s,\omega)|\Big].
\end{equation}
\end{cor}
\begin{proof}
Due to the joint measurability property proved in the previous
proposition, we get for any $T > 0$ via Fubini's Theorem
\[ \int_0^T \EE_i  |\dot\gamma(t,\omega)| \, \dd t =
\EE_i \left [ \int_0^T  |\dot\gamma(t,\omega)| \, \dd t \right ] \]
and the integral in the right--hand side is finite because
\ $\int_0^T  |\dot\gamma(t,\omega)| \, \dd t$\
is bounded uniformly in $\omega$, see Remark \ref{oss admissible
curve}. This implies that the function
\ $t \mapsto \EE_i\Big[
|\dot\gamma(s,\omega)|\Big]$\
is locally summable in $\R_+$, so that by Lebesgue Differentiation
Theorem adapted to the Lebesgue measure, namely not requiring the
shrinking neighborhoods of a given time $s$ to be centered at $s$,
we get
\begin{equation}\label{leb2}
  \lim_{h \to 0^+} \frac 1h \,   \int_s^{s+h} \EE_i \big[ | \dot\gamma(t,\omega)|\big] \, \dd
t = \EE_i\big[ |\dot\gamma(s,\omega)|\big] \quad\hbox{for a.e.
$s\in\R_+$,}
\end{equation}
 again applying Fubini Theorem we have
 \begin{equation}\label{leb3}
    \EE_i \left [ \frac 1h \, \int_s^{s+h}
|\dot\gamma(t,\omega)| \, \dd t \right ]  = \frac 1h \, \int_s^{s+h}
\EE_i | \dot\gamma(t,\omega)| \, \dd t.
\end{equation}
The assertion is a direct consequence of \eqref{leb2}, \eqref{leb3}.
\end{proof}

\end{subsection}

\begin{subsection}{Lipschitz continuous functions and admissible
curves}

\medskip

We proceed by studying the behavior of a Lipschitz continuous function
on an admissible curve.
The first result is

 \begin{prop}\label{absolute} Let $\uu:\R_+ \times \R^N\to\R^m$ be a locally Lipschitz function  and $\gamma$ an admissible curve. For every index $i\in\ind$, the function
\[ t \mapsto \EE_i \big[u_{\omega(t)}\big(t,\gamma(t,\omega)\big) \big] \]
is locally absolutely continuous  in  $\R_+$.
\end{prop}
\begin{proof}  We denote by $f$ the function in object. We fix an
index $i$ and $\eps > 0$.  We consider a bounded interval $I$ and a
finite family of pairwise disjoint intervals $\{(a_j,b_j)\}$
contained in $I$.

Taking into account that $\gamma(0,\omega)$ is constant in
$\Omega_i$ and item (i) in Definition \ref{def admissible curve}, we
see that the curve $\gamma$ lies in a given bounded set $B$ for $t
\in I$ and $\omega \in \Omega_i$. We denote by $R$ an upper bound of
$\uu$ in $I \times B$.  Owing to item (i) in  Definition \ref{def
admissible curve} and  the fact that $\uu$ is locally Lipschitz
continuous, the functions $t \mapsto
u_k\big(t,\gamma(t,\omega)\big)$ are equi--absolutely continuous in
$I$, for $k \in \ind$, $\omega \in \Omega$.  We can therefore
determine a positive constant $\delta$ with
\begin{eqnarray}
   \sum_j(b_j-a_j) <  \delta \; &\Rightarrow& \;  \sum_j \big|u_k\big(b_j,\gamma(b_j,\omega)\big) - u_k\big(a_j,
  \gamma(a_j,\omega)\big)\big| < \frac \eps 2  \label{abso1}\\
  \delta  &<&  \frac 1{2 \, R \,  C }, \label{abso2}
\end{eqnarray}
where $C$ is the constant appearing in the statement of Lemma
\ref{differ}. We claim that
\begin{equation}\label{abso2bis}
\sum_j(b_j-a_j) <  \delta \; \Rightarrow  \; \sum_j|f(b_j) - f(a_j)|
< \eps.
\end{equation}
We set
\[F_j= \{  \omega\mid \omega(a_j) \neq \omega(b_j)\}.\]
We know from Lemma \ref{differ} that
\begin{equation}\label{abso3}
   \PP_i(F_j) \leqslant C \, (b_j-a_j) \qquad\hbox{for any $j$.}
\end{equation}
We have
\begin{eqnarray*}
  \sum_j|f(b_j) - f(a_j)| &\leqslant & \sum_j \EE_i \,
\big[\big|u_{\omega(b_j)}\big(b_j,\gamma(b_j,\omega)\big)-u_{\omega(a_j)}\big(a_j,\gamma(a_j,\omega)\big)\big|
\big ] \\ &\leqslant & \sum_j \int_{\Omega \setminus F_j}
\big[\big|u_{\omega(b_j)}\big(b_j,\gamma(b_j,\omega)\big)-u_{\omega(b_j)}\big(a_j,\gamma(a_j,\omega)\big)\big|
\big] \, \dd \PP_i \\ &+& \int_{
F_j}\big[\big|u_{\omega(b_j)}\big(b_j,\gamma(b_j,\omega)\big)-u_{\omega(a_j)}\big(a_j,\gamma(a_j,\omega)\big)
\big| \big]\, \dd \PP_i
\end{eqnarray*}
 and we conclude, recalling the role of $R$ and \eqref{abso1},
 \eqref{abso2}, \eqref{abso3}
\[\sum_j|f(b_j) - f(a_j)|  <  \frac \eps 2 + 2 \, R \,
C \, \sum_j(b_j-a_j) < \frac \eps 2 +\frac \eps 2 = \eps .\] This
proves \eqref{abso2bis} and concludes the proof.
\end{proof}

\medskip

We go on  proving  that time derivative of a locally Lipschitz
continuous function on an admissible curve  and expectations $\EE_i$
commute, up to a term which, roughly speaking, records the indices
jumps on the underlying paths and contains the coupling matrix.

To comment on it, let us take for simplicity $\gamma$ deterministic
and $\uu$, $\gamma$ both of class $\D{C}^1$. By linearity the
difference quotient of $t \mapsto \EE_i\big[u_{\omega(t)}\big(t,\gamma(t)\big) \big]
$ is given by
\[ \EE_i \Big[ \frac{u_{\omega(t+h)}\big(t+h,\gamma(t+h)\big)-u_{\omega(t)}\big(t,\gamma(t)\big)}h  \Big
].
\]
Owing to right continuity of $\omega$,  the integrand
$\omega$--pointwise converges  to the time derivative of
$u_{\omega(t)}$ on $\gamma$ at $t$  but, due to indices jumps,  it is not bounded in $\Omega$ so
that the Dominated Convergence Theorem cannot be applied to get
the corresponding convergence of  expectations. In this framework the
extra term with the coupling matrix pops up.

This is the main output of the section and    will be exploited to prove
some properties of the Lax--Oleinik formula in Section
\ref{LO}.

\medskip

\begin{teorema}\label{teo key}
Let $\uu:\R_+\times\R^N\to\R^m$ be a locally Lipschitz function and  $\gamma$ an admissible curve. Then, for every index $i\in\ind$, we have
\begin{eqnarray}\label{key00}
    \frac{\dd}{\dd t}\EE_i  \big[u_{\omega(t)}\big(t,\gamma(t,\omega)\big)\big ]_{\mbox{\Large
$|$}_{t=s}} = \EE_i \Big
[-\big(B\uu\big)_{\omega(s)}\big(s,\gamma(s,\omega)\big)  + \frac{\dd}{\dd t}
u_{\omega(s)}\big(t,\gamma(t,\omega)\big)_{\mbox{\Large $|$}_{t=s}}\Big ]
\end{eqnarray}
 for
a.e. $s \in \R_+$.
\end{teorema}

\smallskip

For the proof we need some preliminary material. We consider the map
\[(t,\omega) \mapsto \uu\big(t,\gamma(t,\omega)\big),\]
for some admissible curve $\gamma$. Thanks to the fact that $\uu$ is
(Lipschitz) continuous and $\gamma$ jointly measurable, we derive
that such map is also measurable from $\Bor(\R_+) \otimes \F$ to
$\Bor(\R^m)$. We can therefore argue as in Lemma \ref{lem diff
points} to get:

 \smallskip

 \begin{lemma}\label{lem diff points bis}
 For any locally Lipschitz continuous function $\uu: \R_+ \times \R^N
\to \R^m$ and any admissible curve $\gamma$ the set
\[ \left\{(t,\omega) \mid t \mapsto \uu\big(t,\gamma(t,\omega)\big) \; \hbox{is not
differentiable at $t$}\,\right\} \] belongs to the product
$\sigma$--algebra $\Bor(\R_+)\otimes\F$ and  has vanishing $\leb^1
\times \PP$ measure.
\end{lemma}

\medskip

\begin{proof}[Proof of Theorem \ref {teo key}] \;
  The difference quotient of
 $ t \mapsto \EE_i  \big[u_{\omega(t)}\big(t,\gamma(t)\big)\big]$ at  $s$ is equal
 to $\EE_i \left [ \psi_h(\omega)+ \varphi_h(\omega) \right ]$, with
\begin{eqnarray*}
  \psi_h(\omega) &:=& \frac{u_{\omega(s+h)}\big(s,\gamma(s,\omega)\big)-u_{\omega(s)}\big(s,\gamma(s,\omega)\big)}{h} \\
  \varphi_h(\omega) &:=& \frac{u_{\omega(s+h)}\big(s+h,\gamma(s+h,\omega)\big)
 -u_{\omega(s+h)}\big(s,\gamma(s,\omega)\big)}{h}.
\end{eqnarray*}
Due to right continuity of $\omega$, we further have
\begin{equation}\label{key2}
   \varphi_h(\omega)=\frac{u_{\omega(s)}\big(s+h,\gamma(s+h,
 \omega)\big)-u_{\omega(s)}\big(s,\gamma(s,\omega)\big)}{h},
\end{equation}
for $h >0$ small enough, with smallness depending on $\omega$.

Keeping $s$ frozen, we apply Proposition \ref{prop index derivative}
to $\mathbf g(\omega)= \uu\big(s, \gamma(s, \omega)\big)$, to
 get
\begin{eqnarray}\label{key002}
 \lim_{h\to 0^+} \EE_i \big [ \psi_h(\omega) \big ] =
 \frac{\dd^+}{\dd \hbox{\em t}}\left(\EE_i \big[ u_{\omega(t)}\big(s,\gamma(s,\omega)\big)\big]
 \right )_{\mbox{\Large $|$}_{t=s}}
 =
 -\EE_i \big [ (B\uu\big)_{\omega(s)}\big(s,\gamma(s,\omega)\big) \big ],
\end{eqnarray}
where  the symbol $\frac{\dd^+}{\dd \hbox{\em t}}$ stands for the
right derivative. The assumptions in  Proposition \ref{prop index
derivative} are actually satisfied: in fact  $\omega \mapsto
\uu\big(s, \gamma(s, \omega)\big)$ is bounded in $\Omega_i$ because
of  item (i) in  Definition \ref{def admissible curve}  and the fact
that $\gamma$ has constant value on $\Omega_i$ at $t=0$. It is in
addition $\F_s$--measurable because $\uu$ is continuous and $\gamma$
adapted to the filtration $\{\F_t\}$.

 To handle the  term $\varphi_h$, we
restrict the choice of $s$.  By Lemma  \ref{lem diff points bis}, we
know that the set
\[
N:=\left\{(t,\omega) \mid t \mapsto \uu\big(t,\gamma(t,\omega)\big) \;
\hbox{is not differentiable at $t$}\,\right\}
\]
 has  vanishing  $\leb^1 \times \PP$ measure, and  consequently its
 $s$--sections $N_s$ have vanishing probability for $s$ varying in a set $J$ of full measure in $\R_+$.
 We therefore deduce from \eqref{key2} that
\begin{equation}\label{key10}
    \varphi_h(\omega)\cv{h\to 0^+} \frac \dd{\dd t}
u_{\omega(s)}\big(t,\gamma(t,\omega)\big)_{\mbox{\Large $|$}_{t=s}}
\qquad\hbox{for $s \in J$,  $\omega \in \Omega \setminus N_s$.}
\end{equation}
Due to Corollary \ref{leb}, we can assume, without any loss to
generality, that the $s \in J$ also satisfy the limit relation
\eqref{leb1}. We have for  $h \leqslant 1$
\begin{eqnarray}
  |\varphi_h(\omega)| &=& \Big | \frac{u_{\omega(s+h)}\big(s+h,\gamma(s+h,\omega)\big)
 -u_{\omega(s+h)}\big(s,\gamma(s,\omega)\big)}{h} \Big | \label{key1} \\
 &\leqslant&  1 + \frac 1h \, \kappa \, \big (  \big|\gamma(s+h ,\omega)-
 \gamma(s,\omega) \big| \big ) \leqslant 1 + \frac 1h \, \kappa \, \left (
 \int_s^{s+h} |\dot\gamma(t,\omega)| \, \dd t \right ) \nonumber
\end{eqnarray}
where $\kappa$ a Lipschitz constant for $\uu$ in
 $[s,s+1] \times B$, and $B$ stands for  the bounded set containing the curves $\gamma(\cdot,\omega)$,
 for $t \in [s,s+1]$, as $\omega$ varies in $\Omega$,  see item (i) in Definition
\ref{def admissible curve}.
In addition, by \eqref{leb1}
\begin{equation}\label{key3}
  \EE_i \left [ \frac 1h \, \int_s^{s+h} |\dot\gamma(t,\omega)| \, \dd
t \right ] \qquad\hbox{is convergent as $h \to 0^+$. }
\end{equation}
 Therefore, when $s \in J$,
the sequence $\varphi_h$ is a.e. pointwise convergent thanks to
\eqref{key10}, and dominated  by another sequence with convergent
$\EE_i$ expectation in force of \eqref{key1}, \eqref{key3}. This
allows using the variant of Dominated Convergence Theorem (see for
instance \cite[Theorem 4, Chapter 1.3]{EG}) to get
\begin{equation}\label{key003}
    \lim_{h \to 0^+} \EE_i \left [\varphi_h(\omega) \right ] = \EE_i \left
[\frac \dd{\dd t}
u_{\omega(s)}\big(t,\gamma(t,\omega)\big)_{\mbox{\Large
$|$}_{t=s}}\right ] \quad\hbox{for $s \in J$.}
\end{equation}
Owing to \eqref{key002}, \eqref{key003}, the function
\[  s \mapsto \EE_i  \big[u_{\omega(s)}\big(s,\gamma(s,\omega)\big) \big] = \EE_i [\psi_h(\omega)+ \varphi_h(\omega)]\]
is a.e. right--differentiable in $\R_+$, and  so a.e. differentiable
in view of  Denjoy--Young--Saks Theorem, see Corollary \ref{DYS}.
Formula \eqref{key00} directly comes from \eqref{key002},
\eqref{key003}.
\end{proof}

\medskip

Given the set
\begin{eqnarray}
   \left\{(t,\omega)\in (0,+\infty)\times\Omega \mid t \mapsto \uu(t,\gamma\big(t,\omega)\big)\ \hbox{and}\  t \mapsto \gamma(t,\omega)\; \hbox{are not
differentiable at $t$}\,\right\},  \label{key01}
\end{eqnarray}
we denote by $J$ the set of points $s>0$ such that the $s$--section of the set in
\eqref{key01} has probability $0$ and \eqref{leb1}
holds at $s$. Note that  $J$ has full  measure in $\R_+$ because of Lemma \ref{lem diff points},
Lemma \ref{lem diff points bis} and Corollary \ref{leb}.

 \smallskip

 \begin{lemma}\label{new} Let $\uu$, $\gamma$ be as in Theorem \ref{teo key}
 and let $s \in J$. The compact--valued map
 \[Z(\omega)= \left \{(r,p) \in \partial^C u_{\omega(s)}\big(s, \gamma(s,\omega)\big) \mid r+\langle p,
 \dot\gamma(s,\omega) \rangle = \frac \dd{\dd t} u_\omega(s)\big(s,\gamma(s,\omega)\big) \right \}\]
 is measurable.
 \end{lemma}
 \begin{proof} Since the lengths of the curves $t \mapsto \gamma(t,\omega)$ in $[0,s]$ are equibounded with
 respect to $\omega$, see Remark \ref{oss admissible curve}, and the elements $\gamma(0,\omega)$
 are finite as $\omega$ varies
 in $\Omega$, we deduce  that the set\ $\big\{\big(s,\gamma(s,\omega)\big) \mid \omega \in \Omega\big\}$\
is bounded. We denote by $R$ a  Lipschitz constant for all the $u_i$
in such a set.

We claim that the function
 \begin{equation}\label{new1}
   \omega \mapsto \frac \dd{\dd t} u_{\omega(s)}\big(s, \gamma(s,\omega)\big) \;\;\hbox{is measurable.}
\end{equation}
Indeed, it is obtained as the composition of $\omega\mapsto (\omega,\omega(s))$, which is a measurable map from $(\Omega,\F)$ to $\big(\Omega\times\ind,\F\otimes\parts(\ind)\big)$ by definition of the $\sigma$--algebra $\F$, with $(\omega,i)\mapsto \frac \dd{\dd \hbox{\em t}}
u_i\big(s,\gamma(s,\omega)\big)$, which is a $\Bor([0,s])\otimes\parts\{\ind\}$--measurable real function, as it can be easily checked by
arguing as in Lemma \ref{lem diff points} and
 Proposition \ref{dotga}. Furthermore, since the map $\omega \mapsto \dot\gamma(s, \omega)$
is measurable in force of Proposition \ref{dotga}, we deduce
 that
 \begin{equation}\label{new2}
  \omega \mapsto \langle p, \dot\gamma(s,\omega) \rangle \;\;\hbox{is measurable  for any  $p \in \R^N$.}
\end{equation}
  We proceed by showing that the compact--valued  map
\begin{eqnarray*}
 \widetilde Z(\omega)= \left\{(r,p)\in\R\times\R^N \mid\  |r|+|p| \leqslant R,\   r+\langle p, \dot\gamma(s,\omega)
 \rangle - \frac \dd{\dd t} u_{\omega(s)}\big(s, \gamma(s,\omega)\big) =0 \right  \}
\end{eqnarray*}
is measurable. Taking into account Definition \ref{uscmea}, it is enough to prove that ${\widetilde Z}^{-1} (K) \in
 \F$ for any compact subset $K$ of $\R^{N+1}$. Let $(r_n,p_n)_n$ be a dense
 sequence in $K$, then
 \[{\widetilde Z}^{-1} (K) = \bigcap_{h =1}^\infty \, \bigcup_{n =1}^ \infty
  \left \{\om \mid\  \big|r_n+\langle p_n, \dot\gamma(s,\omega)
 \rangle - \frac \dd{\dd t} u_{\omega(s)}\big(s, \gamma(s,\omega)\big)\big| < \frac
 1h \right \}\]
 and the sets appearing in the above formula belong to $\F$ thanks
 to \eqref{new1}, \eqref{new2}. This concludes the proof of the
 claim.

 Now notice that the set--valued map \ $\omega\mapsto \partial^C u_{\omega(s)}\big(s, \gamma(s,\omega)\big)$\  is $\F$--measurable. Indeed,
 it is obtained as the composition of \ $\omega\mapsto (\gamma(s,\omega),\omega(s))$, which is a measurable map from $(\Omega,\F)$ to $\big(\R^N\times\ind,\Bor(\R^N)\otimes\parts(\ind)\big)$, with $(x,i)\mapsto \partial^C u_i\big(s,x\big)$, which is an uppersemicontinuous set--valued map defined
 on $\R^N\times\ind$.

Bearing in mind the definition of $R$, we
find that
\[
Z(\omega) = \widetilde Z(\omega) \cap \partial^C u_{\omega(s)}\big(s, \gamma(s,\omega)\big).
\]
The map $Z$  has nonempty compact images thanks to Propositions
\ref{convexconvex} and \ref{clarke}, and it is  measurable as
an intersection of measurable set--valued maps, see \cite{CV}.
\end{proof}

\smallskip
We invoke  Theorem \ref{CastaingValadier} to derive

\begin{cor} Let $\uu$, $\gamma$, $s$, $Z(\omega)$  be as in Lemma \ref{new},  then there is a measurable selection
$ \omega \mapsto \big(r(s,\omega), p(s,\omega)\big)$ of $Z(\omega)$.
\end{cor}

\smallskip

Taking into account Theorem \ref{teo key}, Lemma \ref{clarke}, and
the above Corollary,  we get

\smallskip

\begin{cor}\label{postkey} Let $\uu$, $\gamma$, $i$ be as in Theorem \ref{teo key}, and let $s >0$ be in $J$.
Then the map
 $t \mapsto \EE_i \big[u_{\omega(t)}\big(t,\gamma(t,\omega)\big) \big] $ is differentiable at
$s$, and, for every $\omega\in\Omega$, one can  find a measurable
selection  $\big(r(s,\omega), p(s,\omega)\big) \in \partial^C
u_{\omega(s)}\big(s,\gamma(s,\omega)\big)$  satisfying
$$
\frac{\dd}{\dd
t}\EE_i \big[u_{\omega(t)}\big(t,\gamma(t,\omega)\big) \big]_{\mbox{\Large $|$}_{t=s}} =
\EE_i \Big [-\left[B\uu\big(s,\gamma(s,\omega)\big)\right]_{\omega(s)} +
r(s,\omega)+\langle p(s,\omega),\dot\gamma(s,\omega)\rangle \Big ].
$$
\end{cor}

\medskip

We finally state, for the reader's convenience,  a differentiability
property of \quad $t \mapsto \EE_i \big[u_{\omega(t)}\big(t,\gamma(t)\big)\big]$, easily
descending from Theorem \ref{teo key},  in the way we are going to
use it in Theorem \ref{teo domination}.

\smallskip

\begin{cor}\label{postpostkey} Let $\uu$ be a
$\D{C}^1$ function and $\gamma$ a  deterministic curve of class $\D{C}^1$. For every index $i\in\ind$, the map
$t \mapsto \EE_i \big[u_{\omega(t)}\big(t,\gamma(t)\big)\big]$
is right differentiable at $t=0$ and
\begin{eqnarray*}
\frac {\dd^+}{\dd t} \EE_i \big[u_{\omega(t)}\big(t,\gamma(t)\big)\big]_{\mbox{\Large
$|$}_{t=0}} = -\big[B \uu\big(0,\gamma(0)\big)\big]_i +
\partial_t u_i\big(0,\gamma(0)\big)+\big\langle D
u_i\big(0,\gamma(0)\big),\dot\gamma(0)\big\rangle.
\end{eqnarray*}
\end{cor}
\end{subsection}
\medskip
\end{section}

\begin{section}{The random Lax--Oleinik formula and its PDE counterpart}\label{LO}
The random Lax--Oleinik formula is given by
\begin{equation}\label{random Lax-Oleinik formula}\tag{LO}
\big(\S(t)\uu^0\big)_i(x) = \inf_{\gamma(0,\omega)=x} \EE_i\left[
u^0_{\omega(t)}\big(\gamma(t,\omega)\big)+\int_0^t
L_{\omega(s)}\big(\gamma(s,\omega),-\dot\gamma(s,\omega)\big)\,\dd s\right]
\end{equation}
for every $(t,x)\in\cyl$ and $i\in\ind$, and for any bounded initial
datum $\uu^0:\R^N\to\R^m$.  Some few properties can be recovered via direct inspection of the formula.
For every $(t,x)\in\cyl$, we have
\begin{equation}\label{ineq S(t)}
\big ( -\|\uu^0\|_\infty+t\mu \big ) \, \1 \leqslant  \S(t) \uu^0(x)\leqslant \e^{-Bt}\uu^0(x)+tM \, \1 ,
\end{equation}
where $\mu:=\inf\limits_{i}\inf\limits_{\R^N\times\R^N}
 L_i$\quad and\quad $M:=\sup\limits_{i}\sup\limits_{x\in\R^N} L_i(x,0)$.
 The leftmost  inequality in \eqref{ineq S(t)} is immediate,
 while the second follows by taking a constant curve  and by applying \eqref{marziotta}.

We furthermore derive from the definition
\begin{equation}\label{nonexpa}
\|\S(t)\uu^0-\S(t)\vv^0\|_\infty\leqslant \|\uu^0-\vv^0\|_\infty
\qquad\hbox{for $t \geqslant 0$,}
\end{equation}
for any  given pair of bounded functions $\uu^0$, $\vv^0$.

We proceed by introducing a sub--optimality principle  that
will allow us to link \eqref{LO} to systems,  and to show in this way
the semigroup and continuity properties of the related value function.

\smallskip

\begin{definition}\label{def dominated evolution} We say that a function $\uu: \R_+
\times \R^N \to \R^m$  satisfies the  {\em sub--optimality principle} if
\[u_i\big(t_0+h,\gamma(0)\big) - \EE_i \big[u_{\omega(h)}\big(t_0,\gamma(h)\big)\big]\leqslant \EE_i\left[
\int_0^{h} L_{\omega(s)}\big(\gamma(s),-\dot\gamma(s)\big)\,\dd
s\right]
\]
for any   $t_0$, $h \geqslant 0$, $i \in \ind$ and any
deterministic  curve $\gamma$.
\end{definition}
\medskip
The link with the Lax--Oleinik semigroup is given by

\begin{prop}\label{marzia} Let $\uu^0: \R^N \to \R^m$ bounded and continuous.
The function $(t,x) \mapsto \S(t) \uu^0(x)$
    satisfies the sub--optimality principle.
\end{prop}

To prove the proposition, we need some preliminary material. We
denote, for any $h > 0$, by $\Phi_h$  the shift operator defined via
\[\Phi_h(\omega)= \omega(\cdot + h) \qquad\hbox{for any $\omega \in
\Omega$}.\]
We recall that it is a measurable map from $\Omega$ to
$\Omega$, see \cite{SMT14}.

\smallskip

\begin{lemma}\label{premarzia} For any index $i$, any $h >0$ we have
\[\Phi_{h} \# \PP_i = \sum_{k=1}^m \left (\e^{-Bh} \right )_{ik} \,
\PP_k.\]
\end{lemma}
\begin{proof}
If $\mathcal C (t_1, \cdots, t_k; i_1, \cdots i_k)$ is a cylinder,
then
\begin{align*}
\Phi_{h} \# \PP_i \big(\mathcal C (t_1, \cdots, t_k; i_1, \cdots i_k)\big) &= \PP_i \Phi_{h}^{-1}\big(\mathcal
C (t_1, \cdots, t_k; i_1, \cdots i_k)\big)\\
& = \PP_i \big(\mathcal C (t_1+h, \cdots, t_k+h; i_1, \cdots i_k)\big) \\
&= \sum_{k=1}^m \PP_i \big(\mathcal C (h,t_1+h, \cdots, t_k+h; k,i_1, \cdots i_k)\big) \\
&= \sum_{k=1}^m \left (\e^{-Bh} \right )_{ik} \, \PP_k \big(\mathcal
C (t_1, \cdots, t_k; i_1, \cdots i_k)\big).
\end{align*}
The above computation proves the
assertion, because the family of cylinders is a separating class, as it was pointed out
in Section \ref{sez prob}.
\end{proof}

\medskip

\begin{proof}[Proof of Proposition \ref{marzia}]
We select  positive times $h$, $t_0$, a deterministic curve $\gamma$ and $i \in \ind$, we set  $x = \gamma(0)$, $y = \gamma(h)$.
We fix an $\eps >0$ devoted to become infinitesimal, and pick for any $j \in \ind$ an admissible random curve
$\xi_j$, with initial point $y$, such that
\begin{equation}\label{propodomi10}
    (\S(t_0) \uu_0)_j(y) \geqslant \EE_j \left [u^0_{\omega(t_0)}\big(\xi_j(t_0,\theta)\big)+
\int_0^{t_0} L_{\omega(s)}(\xi_j,- \dot\xi_j) \, \dd s \right ]-\eps.
\end{equation}
We proceed by defining for any $(t,\omega) \in [0,+ \infty)
\times \Omega$
\[ \eta(t,\omega) = \left \{ \begin{array}{cc}
                    \gamma(t) & \quad t \in [0,h) \\
                    \xi_{\omega(h)}\big(t-h,\Phi_{h}(\omega)\big) &\quad   t \in [h, t_0+h) \\
                    \xi_{\omega(h)} \big(t_0,\Phi_{h}(\omega)\big) &\quad   t \in [t_0+h, + \infty ) \\
                         \end{array} \right . \]
 We claim  that $\eta$  is an
admissible curve.  We first show  that for any $t \geqslant 0$, any  Borel set $E$ in $\R^N$
\begin{equation}\label{marzia1}
   \Omega_E := \{\omega \mid \eta(t, \omega) \in E\} \in \F_t.
\end{equation}
Clearly $\Omega_E$ is equal either to the whole $\Omega$ or to the empty set when $t \in [0,h]$. We focus on the case where
 $t \in (h, t_0+ h]$, the same argument will give the property when $t > t_0+h$. We have
\begin{equation}\label{marzia2}
\Omega_E= \bigcup_j \big [ \{\omega \mid \xi_j\big(t -h, \Phi_h(\omega)\big)  \in E\} \cap \mathcal C(h,j) \big ].
\end{equation}
Owing to the fact that  $\xi_j$ is $\F_t$ adapted, for any $j$, and to the relation $\Phi_h^{-1}(\F_{t-h}) \subset \F_t$,
see for instance  \cite[Proposition B.5]{SMT14}, we further get
\[ \{\omega \mid \xi_j\big(t-h, \Phi_h(\omega)\big) \in E\} = \Phi_h^{-1} \big (\{\theta \mid
 \xi_j(t-h, \theta) \in E\} \big ) \in \Phi_h^{-1}(\F_{t-h}) \subset \F_t.\]
 This gives \eqref{marzia1} taking into account \eqref{marzia2} and that $\mathcal C(h,j) \in \F_h \subset \F_t$, for any $j$.
  Being the continuous concatenation of a deterministic curve and $m$ random admissible  curves, we also see that $\eta$ satisfies
   item (i) in Definition \ref{def admissible curve}. We have therefore proved that it is an admissible curve, as it was claimed.

We  have by Lax--Oleinik formula
\begin{eqnarray}\label{propodomi2}
&& (\S(t_0+h) \uu^0)_i(x)  \leqslant \EE_i \left [
u^0_{\omega(t_0+h)}\big(\eta(t_0+h,\omega)\big)+\int_0^{t_0+h}
L_{\omega(s)}\big(\eta(s,\omega),-\dot\eta(s,\omega)\big)\,\dd s \right]\nonumber\\
  &&= \EE_i \left [
  u^0_{\Phi_{h}(\omega)(t_0)}\big(\xi_{\Phi_h(\omega)(0)}\big(t_0,\Phi_{h}(\omega)\big)\big)+\int_0^{h}
L_{\omega(s)}\big(\gamma(s),-\dot\gamma(s)\big)\,\dd s \,  +  \right . \\
&& \left .  \int_{0}^{t_0}
L_{\Phi_{h}(\omega)(s)}\Big(\xi_{\Phi_h(\omega)(0)}\big(s,\Phi_{h}
(\omega)\big),-\dot\xi_{\Phi_h(\omega)(0)}\big(s,\Phi_{h} (\omega)\big)\Big)\,\dd
s\right]. \nonumber
\end{eqnarray}
We apply the change of variable formula with $\theta= \Phi_h(\omega)$  and Lemma \ref{premarzia} to get for any $j \in \ind$
\begin{eqnarray*}
   && \int_{\Omega} \Big [
u^0_{\Phi_{h}(\omega)(t_0)}\big(\xi_{\Phi_h(\omega)(0)}\big(t_0,\Phi_{h}(\omega)\big)\big)  \\ &&
+ \int_{0}^{t_0}
L_{\Phi_{h}(\omega)(s)}\Big(\xi_{\Phi_h(\omega)(0)}\big(s,\Phi_{h}
(\omega)\big),-\dot\xi_{\Phi_h(\omega)(0)}\big(s,\Phi_{h} (\omega)\big)\Big)\,\dd
s\Big ] \dd \PP_i(\omega)\\
  &&=\int_{\Omega} \left [
u^0_{\theta(t_0)}\big(\xi_{\theta(0)}(t_0,\theta)\big) + \int_{0}^{t_0}
L_{\theta(s)}\big(\xi_{\theta(0)}(s,\theta),-\dot\xi_{\theta(0)}(s,\theta)\big)\,\dd
s\right ] \dd \Phi_{h} \# \PP_i(\theta) \\
  && = \sum_{j=1}^m     \left (\e^{-Bh} \right )_{ij} \, \EE_j \left [
u^0_{\theta(t_0)}\big(\xi_j(t_0,\theta)\big) + \int_{0}^{t_0}
L_{\theta(s)}\big(\xi_j(s,\theta),-\dot\xi_j (s,\theta)\big)\,\dd
s\right ] .
\end{eqnarray*}
We plug the previous relation in \eqref{propodomi2} and use
\eqref{propodomi10}, \eqref{marziotta} to get
\begin{align*}
& (\S(t_0+h) \uu^0)_i(x)  \leqslant  \EE_i \Big [ \int_0^{h}
L_{\omega(s)}\big(\gamma(s,\om),-\dot\gamma(s,\om)\big)\,\dd s \Big ]  +
\\ &  \sum_{j=1}^m \left (\e^{-Bh} \right )_{ij} \, \EE_j \Big [
u^0_{\theta(t_0)}\big(\xi_j(t_0,\theta)\big) + \int_{0}^{t_0}
L_{\theta(s)}\big(\xi_j(s,\theta),-\dot\xi_j (s,\theta)\big)\,\dd
s \Big ]   \leqslant  \\
& \EE_i \bigg[ \int_0^{h}
L_{\omega(s)}\big(\gamma(s,\om),-\dot\gamma(s,\om)\big)\,\dd s  \bigg ] +
\sum_{j=1}^m \left (\e^{-Bh} \right )_{ij} \, \Big ( (\S(t_0)
\uu_0)_j\big(y \big) + \eps \Big ) = \\
& \EE_i \left [ \int_0^{h}
L_{\omega(s)}\big(\gamma(s,\om),-\dot\gamma(s,\om)\big)\,\dd s +
\big(\S(t_0) \uu_0\big)_{\omega(h)}\big ( y \big)
 + \eps \right ],
\end{align*}
and the assertion follows because $\eps$ is arbitrary.
\end{proof}

\medskip

In the next result we link the   sub--optimality principle  with
the  property of being subsolution to the system \eqref{HJS}. It is
worth pointing out that we are not assuming any continuity or
semicontinuity condition on the function appearing in the statement.

\smallskip

\begin{teorema}\label{teo domination}
Let $\uu: \R_+ \times \R^N \to \R^m$ be a function locally bounded from above satisfying the sub--optimality principle. Then
 it is a subsolution to \eqref{HJS} in $(0,+\infty) \times \R^N$.
\end{teorema}

\begin{proof}
Recall that we denote by $\uu^*$ the upper semicontinuous envelope
of $\uu$.

Let us fix an index $i\in\ind$, $q \in \R^N$,   $\delta>0$. Let
$\psi$ be a $\D{C}^1$ supertangent to $u^*_i$ at a point $(s_0,x_0)
\in (0,+\infty) \times \R^N$.  We define a $\D{C}^1$, $\R^m$--valued
function $\phi$ setting $\phi_i=\psi$,  and choosing the other
components $\phi_j$ of class $\D{C}^1$ such that
\begin{equation}\label{domina00}
   \phi_j(s_0,x_0)= u^*_j(s_0,x_0) + \delta \qquad\hbox{for $j \neq
   i$.}
\end{equation}
The  definition implies
\begin{equation}\label{domina1}
   u^*_i(s_0,x_0) = \psi(s_0,x_0)= \phi_i(s_0,x_0)
\end{equation}
and
\begin{equation}\label{domina2}
  \phi_k \geqslant  u^*_k  \geqslant u_k \qquad\hbox{ for any $k \in \ind$, in some
neighborhood of $(s_0,x_0)$.}
\end{equation}
 The argument being local, we can assume $\phi$ uniformly continuous with corresponding modulus denoted by $\nu$.
We consider $(s_n,x_n)$ converging to
 $(s_0,x_0)$ with $\lim_n u_i(s_n,x_n)= u_i^*(s_0,x_0)$, and  set
 \[\eps_n = \nu(|x_n -x_0| + |s_n - s_0|) + |u_i(s_n,x_n)-
 u_i^*(s_0,x_0)|.\]
We  proceed by  selecting an infinitesimal positive sequence $h_n$ with
\begin{equation}\label{domina5}
 \lim_n \frac{\eps_n}{h_n}=0
\end{equation}
 and   define $\gamma_n(s)= x_n + s \, q$, $\gamma(s)= x_0 + s \,
q$. From the supertangency condition, the relation $|\gamma_n(s)-
\gamma(s)| = |x_n - x_0|$ for any $s$, and the definition of
$\eps_n$, we derive for $n$ large enough
\begin{align}
  \psi(s_0,x_0) &- \EE_i
\big[\phi_{\omega(h_n)}\big(-h_n+s_0,\gamma(h_n)\big)\big] \leqslant  \nonumber \\
 &\leqslant
 u_i(s_n,x_n) - \EE_i \big[\phi_{\omega(h_n)}\big(-  h_n +s_n, \gamma_n(h_n)\big)\big] + \eps_n        \nonumber  \\
& \leqslant    u_i(s_n,x_n) - \EE_i \big[u_{\omega(h_n)}\big(-  h_n +s_n, \gamma_n(h_n)\big)\big] +  \eps_n.   \label{domina3}
\end{align}
Further, we have by the sub--optimality principle
\begin{equation}\label{domina6}
   u_i(s_n,x_n) - \EE_i \big[u_{\omega(h)}\big(-  h_n +s_n, \gamma_n(h_n)\big)\big]  \leqslant
  \EE_i\left [ \int_0^{h_n} L_{\omega(s)}  (\gamma_n(s),-
 q ) \,\dd s   \right].
\end{equation}
We apply Corollary \ref{postpostkey} to the function $\widetilde
\phi(s,x) := \phi(s_0-s,x)$ and the curve
 $\gamma$. We derive from \eqref{domina00}, \eqref{domina1}, \eqref{domina2}, \eqref{domina5}, \eqref{domina3}, \eqref{domina6},  and assumption (B) on the
coupling matrix,
\begin{align}
&\liminf_{n\to+\infty} \frac 1{h_n} \EE_i\left [ \int_0^{h_n}
L_{\omega(s)}  (\gamma_n(s),-
 q ) \,\dd s   + \eps_n \right]
\nonumber
\\ &\geqslant  \lim_{n\to+\infty} \frac 1{h_n} \, \big (
\psi(s_0,x_0)
 - \EE_i \big[\phi_{\omega(h_n)}\big(-h_n+s_0,\gamma(h_n)\big)\big] \big ) \nonumber
\\ & = - \frac {\dd^+}{\dd t} \EE_i \big[\widetilde \phi
_{\omega(t)}\big(t,\gamma(t)\big)\big]_{\mbox{\Large $|$}_{t=0}} \label{domina4} \\
& = \big[B \widetilde \phi\big(0,\gamma(0)\big)\big]_i -
\partial_t \widetilde \phi_i\big(0,\gamma(0)\big)- \big\langle D_x
\widetilde \phi_i\big(0,\gamma(0)\big),\dot\gamma(0)\big\rangle. \nonumber\\
 &= \big[B \phi(s_0,x_0)\big]_i
+ \partial_t \psi(s_0,x_0)-\langle D_x\psi(s_0,x_0),q\rangle \nonumber
\\ &\geqslant  \big [ B \uu^*(s_0,x_0)\big ]_i
+ \partial_t \psi(s_0,x_0) + \langle D_x\psi(s_0,x_0),- q\rangle -
\delta \, C, \nonumber
\end{align}
where $C =  - \sum_{j \neq i} b_{ij}$. We know that   $\PP_i$--a.e.
path  $\omega$   takes the value $i$ in a suitable right
neighborhood of $0$, depending on $\omega$. From this and the
continuity of $L_i$,  we deduce for $\PP_i$--a.e. $\omega$
\begin{eqnarray*}
   && \lim_n \frac 1{h_n} \, \int_0^{h_n}
L_{\omega(s)}  (\gamma_n(s),- q )\,\dd s  =  \\
   && \lim_n \frac 1{h_n} \, \int_0^{h_n}
L_i  (\gamma_n(s),-  q )\,\dd s \\ && \lim_n \frac 1{h_n} \,
\int_0^{h_n} L_i  (\gamma(s),-  q )\,\dd s +  \frac 1{h_n} \,
\int_0^{h_n} \big ( L_i (\gamma_n(s),-  q ) - L_i(\gamma(s), - q)
\big )\,\dd s \\ &&  = L_i(x_0,- q).
\end{eqnarray*}
By  the Dominated Convergence Theorem and \eqref{domina5}, we thus
infer
\[\lim_{n\to +\infty} \frac 1{h_n} \,
 \EE_i\left[ \frac 1{h_n} \, \int_0^{h_n}
L_{\omega(s)}  (\gamma_n(s),- q )\,\dd s  + \eps_n\right] =
L_i(x_0,- q) .\]
 We further derive from \eqref{domina5}
\[ \big(B \uu^*(s_0,x_0)\big)_i  + \partial_t
\psi(s_0,x_0)+\langle D_x\psi(s_0,x_0),- q\rangle - L_i(x_0,-q) -
\delta \, C \leqslant 0,\]  and, being  $q$, $\delta$  arbitrary, we
finally obtain
\[ \big(B\uu^*(s_0,x_0)\big)_i  + \partial_t
\psi(s_0,x_0)+ H_i\big(x_0,D_x \psi(s_0,x_0)\big)\leqslant 0,\] which proves the
claimed subsolution property for $\uu^*$.
\end{proof}

\smallskip

By combining   Proposition \ref{marzia} and Theorem \ref{teo domination} we obtain

\smallskip

\begin{prop}\label{prop domination}
Let $\uu^0:\R^N\to\R^m$ be  bounded and continuous. The function
$\vv(t,x):= \S(t) \uu^0(x)$ is a  subsolution of
\eqref{HJS} satisfying $\vv^*(0,\cdot) \leqslant \uu^0$ on $\R^N$.
\end{prop}
\begin{proof} The  asserted subsolution property comes from  $\vv$ being locally bounded in force of \eqref{ineq S(t)},
 and  Proposition \ref{marzia},  Theorem \ref{teo domination}. The relation  at $t =0$ is readily obtained passing to the limit
 in the rightmost inequality of  \eqref{ineq S(t)} as $t$ goes to $0$.
\end{proof}

\smallskip

We proceed by showing a sort of maximality property of the function given by Lax--Oleinik formula.

\smallskip

\begin{prop}\label{pro dominationbis}
Let $\uu$ be a locally Lipschitz continuous subsolution of
\eqref{HJS}  with $\uu^0:= \uu(0,\cdot)$ bounded, then
\[ \uu(t,x) \leqslant \S(t) \uu^0(x) \qquad\hbox{for any $(t,x) \in \R_+ \times \R^N$.}\]
\end{prop}

\begin{proof}
 We fix $(t_0,x_0)$ and pick an admissible curve $\gamma$  with initial point $x$. By
applying Corollary \ref{postkey}, we get
\begin{align*}
&\frac{\dd}{\dd
t}\EE_i \left [u_{\omega(t)}\big(t_0 -t,\gamma(t,\omega)\big) \right ]_{\mbox{\Large
$|$}_{t=s}}\\
&\qquad\qquad= \EE_i \left
[-\big[B\uu\big(t_0 -s,\gamma(s,\omega)\big)\big]_{\omega(s)} - r(s,\omega)+\langle
p(s),\dot\gamma(s,\omega)\rangle \right ]
 \end{align*}
for a.e. $s$ and some $\big(r(s,\omega),p(s,\omega)\big) \in \partial^C
u_{\omega(s)}\big(t_0+h-s, \gamma(s,\omega)\big)$. By exploiting the subsolution
property of $\uu$ and the Fenchel inequality, we further get
\begin{eqnarray}
&-&\! \! \!  \frac{\dd}{\dd
t}\,\EE_i \big[u_{\omega(t)}\big(t_0 -t,\gamma(t,\omega)\big) \big]_{\mbox{\Large $|$}_{t=s}}   \nonumber\\
&=& \EE_i \left [\big[B\uu\big(t_0  -s,\gamma(s,\omega)\big)\big]_{\omega(s)} +
r(s,\omega) +\langle  p(s,\omega),- \dot\gamma(s,\omega)\rangle \right ] \nonumber\\
  &\leqslant&\! \EE_i \Big
[\big[B\uu(t_0 -s,\gamma(s,\omega)\big)\big]_{\omega(s)} + r(s,\omega) +
H_{\omega(s)}\big(\gamma(s,\omega),p(s,\omega)\big) \nonumber
\\
&&\qquad+  L_{\omega(s)}\big(\gamma(s,\omega),- \dot\gamma(s,\omega)\big)\Big ]   \nonumber \\
&\leqslant&
 \EE_i \left [L_{\omega(s)}\big(\gamma(s,\omega),- \dot\gamma(s,\omega)\big)\right]. \label{Fenchel}
\end{eqnarray}
We finally obtain  by integrating between $0$ and $t_0$  and by commuting integrals, which can be
done by joint measurability properties of $(\gamma,\dot\gamma)$
\[ u_i\big(t_0 ,x_0\big) - \EE_i  \big[u^0_{\omega(t_0)}\big(\gamma(t_0,\omega)\big)\big] \leqslant
\EE_i\left[ \int_0^{t_0}
L_{\omega(s)}\big(\gamma(s,\omega),-\dot\gamma(s,\omega)\big)   \,\dd s \right].\]
This gives the assertion for $\gamma$ is an arbitrary admissible curve starting at $x_0$.
\end{proof}

\

We finally provide the announced PDE characterization of the  random Lax--Oleinik formula.

\begin{teorema}
Let $\uu^0\in\big(\D{BUC}(\R^N)\big)^m$. Then
\ $(x,t) \mapsto \big(\S(t)\uu^0\big)(x)$\
is the unique solution of\eqref{HJS} in $(0.+ \infty) \times \R^N$ agreeing with $\uu^0$ at $t =0$ and belonging to
 $\big(\D{BUC}(\clTcyl)\big)^m$ for every $T>0$.
\end{teorema}

\begin{proof}
We denote by $\uu$ the unique solution of the system taking the value $\uu^0$ at $t =0$ and belonging to
 $\big(\D{BUC}(\clTcyl)\big)^m$ for every $T>0$, see Theorem \ref{teo existence evo wcs}.

Let us first assume $\uu^0$ Lipschitz continuous on $\R^N$. Then $\uu$ is Lipschitz continuous in $\clTcyl$, for every $T >0$,
according to Theorem \ref{teo existence evo wcs}. In view of Proposition \ref{prop domination}  and the comparison principle for \eqref{HJS} stated
in Proposition \ref{prop comparison},  we infer
\begin{eqnarray*}
\big(\S(t)\uu^0\big)(x)\leqslant \uu(t,x)\qquad\hbox{for every $(t,x)\in\clTcyl$.}
\end{eqnarray*}
The opposite inequality holds as well by Proposition \ref{pro dominationbis}. The assertion is then proved when the initial datum
is additionally assumed Lipschitz continuous.

Let us now consider the general case $\uu^{0}\in\vBUC$. Let $\vv^0$ be a Lipschitz function in $\vBUC$.
By what was just proved, we know that the map $(t,x)\mapsto\big(\S(t)\vv^0\big)(x)$ is a Lipschitz solution of \eqref{HJS}
in $\Tcyl$ taking the initial value $\vv^0$ at $t=0$.
From the comparison principle stated in Proposition \ref{prop comparison}  and \eqref{nonexpa}, we infer
\begin{eqnarray*}
\|\uu(t,\cdot)-\S(t)\uu^0\|_\infty \leqslant \|\uu(t,\cdot)-\S(t)\vv^0\|_\infty +
\|\S(t)\uu^0-\S(t)\vv^0\|_\infty\leqslant 2 \,  \|\uu^0-\vv^0\|_\infty
\end{eqnarray*}
for every $t>0$.  Using the fact that Lipschitz initial data are dense in $\vBUC$, we eventually
get the asserted identity \ $\uu(t,x)=\S(t)\uu^0(x)$\ for every
$(t,x)\in [0, + \infty)\times\R^N$.
\end{proof}

\smallskip

We  directly derive  from the previous results of the section
\smallskip

\begin{cor}
The Lax--Oleinik formula defines a semigroup of operators on both $\big(\D{BUC}(\R^N)\big)^m$ and
$\left(\D{Lip}(\R^N)\right)^m \bigcap \big(\D{BUC}(\R^N)\big)^m$.
\end{cor}
\medskip
\end{section}

\begin{section}{Minimal admissible curves}\label{minimo}

In this section we aim to prove the following result:
\smallskip

\begin{teorema}\label{butterfly}
Let $\uu^0\in\vBUC$  and $T>0$. Assume that the function $(t,x)
\mapsto  \S(t)\uu^0(x)$ is locally Lipschitz in $(0,T] \times \R^N$.
Then, for every  $x \in \R^N$, $i\in\ind$,   there  exists an
admissible curve $\eta:\Omega\to\curves$, starting at $x$,
realizing the minimum for $(\S(T)\uu^0)_i(x)$ in the Lax--Oleinik
formula \eqref{random Lax-Oleinik formula}.
\end{teorema}

\begin{oss}
The assumption of the above theorem is always satisfied whenever $\uu^0$
is Lipschitz continuous on $\R^N$, in view of Theorem \ref{teo existence evo wcs}.\\
\end{oss}

\subsection{Deterministic minimization}\label{sez deterministic minimization}
Let $\uu^0\in\vBUC$ and $T>0$ be fixed, and denote by $\uu(t,x)$ the unique function in $\tvBUC$ that solves
the system \eqref{HJS} in $\Tcyl$ subject to the initial condition $\uu(0,\cdot)=\uu^0$ in $\R^N$. We know that,
for every $j\in\ind$, the $j$--th component $u_j$ of $\uu$ is a solution to
\begin{equation}\label{eq G}
\frac{\partial u}{\partial t} + G_j(t,x,D_x u) = 0\qquad\hbox{in
$(0,+\infty)\times\R^N$}.
\end{equation}
with initial datum $u^0_j$, where
\[G_j(t,x,p) = H_j(x,p) +  (B\uu)_j(t,x).
\]
 Let us denote by $L_{G_j}= L_j - (B\uu)_j$ the Lagrangian associated with $G_j$ via the Fenchel transform.
%, i.e.
% \begin{equation*}
% L_{G_j}(t,x,q):=\sup_{p\in\R^N}\left\{\langle p, q\rangle -
% G_j(t,x,p)\right\}.
% \end{equation*}
%
The following result holds:

\begin{prop}\label{prop deterministic Lax}
Let $u_j$ and $G_j$ be as above. Then, for every $0\leqslant a
\leqslant T$ and $y\in\R^N$, the following identity holds:
\begin{equation}\label{eq Lax-Oleinik GG}
u_j(T-a,y)=\inf_{\xi(a)=y}\Big( u^{0}_j\big(\xi(T)\big)+\int_a^T L_{G_j}\big(T-t,\xi (t), -\dot \xi (t) \big) \dd t\Big),
\end{equation}
where the infimum is taken by letting $\xi$ vary in the family of absolutely continuous curves from $[a,T]$ to $\R^N$. Moreover, such an infimum is a minimum.
\end{prop}
\begin{oss}\rm
Note that there is a slight difference between \eqref{eq Lax-Oleinik
GG} and the other deterministic formula given in
 Proposition \ref{L.O.t} of the appendix. However, both formulas are equivalent up to the change of variables $s=T-t$.
\end{oss}

\begin{proof}
Let us first assume $\uu^0$ Lipschitz continuous, so that $u_j$ is Lipschitz in $\clTcyl$.
Then the result is a direct application of Proposition \ref{L.O.t}.

In the general case, let $\vv^n : [0,T]\times \R^N \to \R^m$ be a sequence of solutions of \eqref{HJS} with Lipschitz initial data $\vv^n(0,\cdot)$, uniformly converging to $\uu^0$. By Proposition \ref{prop comparison}, $\vv^n$ uniformly converges to $\uu$ on $[0,T]\times \R^N$. Denote by   $U_j(T-a,y)$ the right hand side of  \eqref{eq Lax-Oleinik GG} and denote by $L_j^n$ the Lagrangian $L_j - (B\vv^n)_j$. It is readily verified from the formula and the first part of the proof that,  for
$0\leqslant a\leqslant T$,
\begin{eqnarray*}
|U_j(T-a,y) - v_j^n(T-a,y)|   \leqslant \|\uu_0 - \vv^n(0,\cdot) \|_{\infty} + (T-a)\| L_{G_j} - L_j^n \|_{L^\infty([0,T]\times\R^N)}    \\
  = \|\uu_0 - \vv^n(0,\cdot) \|_\infty + (T-a)\| (B\uu - B\vv^n)_j \|_{L^\infty([0,T]\times\R^N)}.
\end{eqnarray*}
It follows that $\vv^n$ uniformly converges to $\mathbb U :=
(U_j)_{1\leqslant j\leqslant m}$, hence $\mathbb U=\uu$, as it was
to be shown.
\end{proof}

Given  $0\leqslant a \leqslant T$, $y\in\R^N$ and $j\in\ind$, let us
denote by $\Gamma_j(a,y)$ the family of absolutely continuous curves
$\xi:[0,T]\to\R^N$ such that $\xi(s)=y$ for every $s\in [0,a]$ and
$\xi_{{\large |}_{[a,T]}}$ realizes the infimum in \eqref{eq
Lax-Oleinik GG}. In what follows, the space $\curves$ of continuous
curves from the interval $[0,T]$ to $\R^N$ is  endowed with the
uniform norm, which makes it a Polish space, and the corresponding
Borel $\sigma$--algebra. The following holds:

\smallskip

\begin{prop}\label{prop deterministic curves}
Let $j\in\ind$. Then
\begin{itemize}
 \item[(i)] the set\quad $X^T_j:=\{\xi\,:\,\xi\in\Gamma_j(a,y)\ \hbox{for some $(a,y)\in [0,T]\times\R^N$}\,\}$
 is a family of equi-absolutely continuous curves in $\curves$;\smallskip
 \item[(ii)]
 $\Gamma_j(a,y)$ is a compact subset of $\curves$, for every $(a,y)\in (0,T)\times\R^N$;\smallskip
 \item[(iii)] the set--valued map $(a,y)\mapsto \Gamma_j(a,y)$ from $[0,T)\times\R^N$ to $\curves$ is
 upper semicontinuous in the sense of Definition \ref{uscmea};\smallskip
 \item[(iv)] for every $y\in\R^N$, $0\leqslant a\leqslant T$ and $\xi\in\Gamma_j(a,y)$ we have
 \begin{eqnarray*}
    u_j\big(T-t,\xi(t)\big)=u^{0}_j\big(\xi(T)\big)+\int_{t}^T L_{G_j}\big(T-s,\xi (s), -\dot \xi (s) \big) \dd s
 \end{eqnarray*}
{for every $t\in [a,T]$.}
\end{itemize}
\end{prop}

\begin{proof}
The first point is  standard in Calculus of Variations and is the
first step in establishing Tonelli's existence Theorem. Let us
denote by $\Theta:\R_+\to\R$ a superlinear function such that
\[
L_i(x,q)\geqslant \Theta(|q|)\qquad\hbox{for every $(x,q)\in\R^N\times\R^N$ and $i\in\ind$.}
\]
Such a function $\Theta$ does exist for the Hamiltonians $H_i$
satisfy condition (H3). Since any $\xi\in X_j^T$ is a minimizer of
\eqref{eq Lax-Oleinik GG} for some $(a,y)\in [0,T]\times\R^N$, we
infer
$$
\int_a^T \Theta\big(|\dot\xi(t)|\big)\,\dd t \leqslant \int_a^T
L_{G_j}\big(T-t,\xi (t), -\dot \xi (t) \big) \dd t \leqslant
2\|u_j\|_\infty,$$ yielding
\[
\int_0^T \Theta\big(|\dot\xi(t)|\big)\,\dd t = \int_0^a
\Theta\big(|\dot\xi(t)|\big)\,\dd t + \int_a^T
\Theta\big(|\dot\xi(t)|\big)\,\dd t \leqslant |\Theta(0)| T +2
\|u_j\|_\infty.
\]
This readily implies (i) in view of \cite[Theorem 2.12]{BGH}.

We will prove items (ii) and (iii) by using  Arzel\`a-Ascoli,
Dunford--Pettis theorems (notice that by (i) the elements of $X^T_j$
are equicontinuous) and  the lower semicontinuity
 of $\xi \mapsto \int_a^T L_{G_j}\big(T-t,\xi (t), -\dot \xi (t) \big) \dd t$, see \cite[Theorem 3.6]{BGH}.

Let us prove (iii) first. We have to check that $\Gamma_j$ satisfies
Definition \ref{uscmea}.  Let $C$ be a closed subset of $\curves$,
$(a_n,y_n)$ a sequence of $\Gamma_j^{-1}(C)$ converging to some
$(a,y)$. We consider a sequence  $\xi_n \in \Gamma_j(a_n,y_n) \cap
C$. By the first part of the proof, we can assume that $\xi_n$
converges, up to extracting a subsequence, to some $\xi$ uniformly
in $[0,T]$ and that $\dot \xi_n$ converges to $\dot \xi$ weakly in
$\big(L^1\left([0,T]\right)\big)^N$, see \cite[Theorem 2.13]{BGH}.
Hence, passing to the limit in the equalities
$$u_j(T-a_n,y_n) - u^{0}_j\big(\xi_n(T)\big)= \int_{a_n}^T L_{G_j}\big(T-t,\xi_n (t), -\dot \xi_n (t) \big) \dd t,
$$
and using the lower semicontinuity of the integral functional, it follows that
$$u_j(T-a,y) - u^{0}_j\big(\xi(T)\big)\geqslant \int_a^T L_{G_j}\big(T-t,\xi (t), -\dot \xi (t) \big) \dd t.
$$
Since \eqref{eq Lax-Oleinik GG} gives the reverse inequality, we
deduce that $\xi$ belongs to $\Gamma_j(a,y)$ and clearly also to
$C$. Therefore $(a,y) \in \Gamma_j^{-1}(C)$,  so that
$\Gamma_j^{-1}(C)$ is closed. This shows assertion (iii). Item (ii)
 follows arguing as above and  taking $a_n=a$ and $y_n=y$ for every $n\in\N$.

Item (iv) is a consequence of the sub--optimality principle, see
Proposition \ref{L.O.t}. Indeed, we have
\begin{align*}
&u_j\big(T-a,\xi(a)\big)-u^{0}_j\big(\xi(T)\big)\\
&= \Big(u_j\big(T-a,\xi(a)\big)-u_j\big(T-t,\xi(t)\big)\Big)
 +\Big(u_j\big(T-t,\xi(t)\big)-u^{0}_j\big(\xi(T)\big)\Big) \\
&\leqslant \int_a^t L_{G_j}\big(T-s,\xi (s), -\dot \xi (s) \big)\, \dd s
+ \int_t^T L_{G_j}\big(T-s,\xi (s), -\dot \xi (s) \big)\, \dd s.
\end{align*}
The previous inequality, which is a sum of two inequalities, is actually an equality in view of \eqref{eq Lax-Oleinik GG}. Hence both inequalities were equalities to start with, as it was to be proved.
\end{proof}

We define a set--valued map \
$\Gamma:[0,T)\times\R^N\times\ind\to\curves$ \ by setting
$\Gamma(a,y,j):=\Gamma_j(a,y)$.  It is compact--valued and upper
semicontinuous, hence measurable. We are thus in the hypotheses of
Theorem \ref{CastaingValadier}, so that  there exists a measurable
selection $\Xi$ for $\Gamma$, i.e. a measurable function
\begin{equation}\label{def Theta}
 \Xi:[0,T]\times\R^N\times\ind\to\curves
\end{equation}
such that
\[\Xi(a,y,j) \in \Gamma(a,y,j) \qquad\hbox{for every $(a,y,j)\in
[0,T]\times\R^N\times\ind$.}\]
Notice that the measurability condition can be equivalently
rephrased requiring
 \[(t,a,y,j) \mapsto \Xi(a,y,j)(t)\]
 to be measurable from $[0,T]\times [0,T]\times \R^N\times\ind$ to $\R^N$
with the natural Borel $\sigma$--algebras.\\

\bigskip

\subsection{Random minimization}\label{randomM}

Given given $T >0$, $x \in \R^N$, $\uu^0\in\vBUC$, $i\in \ind$, we aim to
show the existence of a minimizing admissible curve $\eta$ for
$(\S(T)\uu^0)_i(x)$. We will provide  a  rather explicit
construction via concatenation of minimizers of  \eqref{eq
Lax-Oleinik GG}.

We first give a rough picture of it, just to contribute some
insight. We fix $\omega$ and minimize at the initial step the
deterministic functional in \eqref{eq Lax-Oleinik GG} in the
interval $[a,T]$ with $a =0$, $j = \omega(0)$ and $y = x$. Clearly
we get multiple minimizers, but, according to the results of the
previous subsection, we can select one  in a measurable way with
respect $(a,y,j)$. This will be crucial to show the random character
of the curve obtained as output.
At the first jump time of $\omega$,  say $\tau_1(\omega)$, we switch to
index $j$ accordingly and restart the procedure minimizing  on
$[\tau_1(\omega), T]$,  with $y$ equal to the position reached by
the selected minimizing curve in the previous step at $\tau_1$.
Notice that the final time   $T$ stays untouched, which is needed to
get in the end  a nonanticipating random curve. We go on iterating
the above procedure at any jump time of $\omega$ belonging to
$[0,T]$.

We proceed by presenting a full description of the construction. We point out that
this part is independent of the local Lipschitz continuity assumption on
$(x,t) \mapsto \mathcal S(t) \mathbf u^0(x)$. This condition  will come into play
only in the proof of Theorem \ref{butterfly}, to apply the
derivation formula given in Theorem \ref{teo key}.

For any fixed $\omega\in\Omega$, we set $\tau_0(\omega)=0$ and we define inductively a sequence $\big(\tau_k(\omega)\big)_k$ by setting
\begin{equation*}
\tau_k(\omega)=
\begin{cases}
 \hbox{$k$--th jump time} & \hbox{ if $\omega$ has at least
  $k$ jump times in $[0,T]$}\\
T & \hbox{otherwise}\\
\end{cases}
\end{equation*}
%
%\tau_k(\omega):=T\wedge\inf\left\{ t\in [\tau_{k-1}(\omega),T)\,:\,\omega(t)\not=\omega\big(\tau_{k-1}(\omega)\big)\,\right\}.
% In the above formula, we agree that the infimum of an empty set is $+\infty$.
%
Let us denote by $\Xi:[0,T]\times\R^N\times\ind\to\curves$ the measurable selection introduced at the end of Section \ref{sez deterministic minimization}, and define inductively a sequence $\big(x_k(\omega)\big)_{k\geqslant 0}$ of points in $\R^N$ by setting $x_0=x$ and
\[
x_k(\omega)=\Xi\big(\tau_{k-1}(\omega),x_{k-1}(\omega),\omega(\tau_{k-1}(\omega)\big)\big(\tau_k(\omega)\big)\quad\hbox{for every $k\geqslant 1$}.
\]
The following holds:
\begin{lemma}
For every $k\in\N$, the maps $\omega\mapsto \tau_k(\omega)$, \ $\omega\mapsto\omega\big(\tau_k(\omega)\big)$ and $\omega\mapsto x_k(\omega)$ are random variables.
\end{lemma}

\begin{proof}
We start by proving the assertion for the maps $\omega\mapsto\tau_k(\omega)$ and $\omega\mapsto\omega(\tau_{k-1}\big(\omega)\big)$. The argument is by induction on $k\geqslant 1$. Let us denote by $(t_n)_{n\in\N}$ a dense sequence in
$(0,T)$. For every $t\in (0,T]$, we have
\[
\{\omega\,:\,\tau_1(\omega) < t\}= \bigcup_{t_n<t} \{\omega(t_n) \neq \omega(0)\} \in \F,
\]
which clearly gives the asserted measurability of $\omega\mapsto\tau_1(\omega)$ (note that $\tau_k(\omega)\leqslant T$ for every $k\geqslant 1$). The fact that $\omega\mapsto\omega\big(\tau_0(\omega)\big)=\omega(0)$ is a random variable is trivial by the definition of $\F$.
Assume now that $\omega\mapsto\tau_j(\omega)$ and $\omega\mapsto\omega\big(\tau_{j-1}(\omega)\big)$ are random variables for every $j \leqslant k$. By the inductive step we have that, for every $t\in (0,T)$,
\[
\{\omega\,:\,\tau_{k+1}(\omega) < t\}= \bigcup_{t_n<t}  \{\omega\,:\,\tau_k(\omega) < t_n\} \cap
\big\{\omega\,:\,\omega(t_n) \neq \omega\big(\tau_{k}(\omega)\big)\big\} \in \F,
\]
thus showing the asserted measurability of $\omega\mapsto\tau_{k+1}(\omega)$.
The fact that $\omega\mapsto\omega\big(\tau_k(\omega)\big)$ is a random variable follows from the fact that, for every $i\in\ind$,
\[
\big \{\omega\,:\,\omega\big(\tau_k(\omega)\big)=i\,\big\}=\bigcup_n \{\omega\,:\,\omega(t_n)=i\,\}\cap \{\omega\,:\,\tau_{k}(\omega)\leqslant t_n<\tau_{k+1}(\omega)\,\}.
\]
Last, the fact that the map $\omega\mapsto x_k(\omega)$ is a random variable for every $k\geqslant 0$ is again by induction on $k$. The measurability for $k=0$ is trivial. Let us assume that $\omega\mapsto x_j(\omega)$ is measurable for every $j\leqslant k$. Then the map $\omega\mapsto x_{k+1}(\omega)$ is a random variable since it is the composition of the
$\F$--measurable function $\omega\mapsto \Big(\tau_{k+1}(\omega),\tau_{k}(\omega),x_k(\omega),\omega\big(\tau_{k}(\omega)\big)\Big)$ with $(t,a,y,j) \mapsto \Xi(a,y,j)(t)$, which is
measurable from $[0,T]\times [0,T]\times \R^N\times\ind$ to $\R^N$ with
 the natural Borel $\sigma$--algebras.
\end{proof}

The sought curve is defined by setting $\eta(0,\omega)=x$ and, for every $k\geqslant 0$,
\begin{eqnarray*}
\eta(t,\omega):=\Xi\Big(\tau_{k}(\omega),x_{k}(\omega),\omega\big(\tau_{k}(\omega)\big)\Big)(t\wedge T)\qquad\hbox{if \ \ $t\wedge T\in (\tau_{k}(\omega),\tau_{k+1}(\omega)]$},
\end{eqnarray*}
for every $\omega\in\Omega$, where $t\wedge T:=\min\{t,T\}$. Note that the curve $\eta(\cdot,\omega)$ is constant in $[T,+\infty)$, for any fixed $\omega\in\Omega$. The following holds:

\begin{prop}
The curve $\eta:\Omega\to \D C\big(\R_+;\R^N\big)$ is admissible.
\end{prop}

\begin{proof}
For every fixed $\omega$, the map $t\mapsto\eta(t,\omega)$ is constructed as a concatenation of equi--absolutely continuous curves, see Proposition \ref{prop deterministic curves}, so it is clear that $\eta$ satisfies item (i) of Definition \ref{def admissible curve}. Its non--anticipating character is also clear by definition. It is left to show that $(t,\omega)\mapsto \eta(t,\omega)$ is jointly measurable from $\R_+\times\Omega$ to $\R^N$ with respect to the product $\sigma$--algebra $\Bor(\R_+)\otimes\F$. To this aim, we remark that
\begin{eqnarray*}
 \eta(t,\omega)=x\cchi_{\{0\}}(t)+\sum_{k=0}^{+\infty}\Xi\Big(\tau_{k}(\omega),x_{k}(\omega),\omega\big(\tau_{k}(\omega)\big)\Big)(t\wedge T)\,\cchi_{(\tau_{k}(\omega),\tau_{k+1}(\omega)]}(t\wedge T),
\end{eqnarray*}
where we agree that the characteristic function $\cchi_\varnothing(\cdot)$ of the empty set is identically 0.
For each $k\geqslant 0$, the map $(t,\omega)\mapsto \Xi\Big(\tau_{k}(\omega),x_{k}(\omega),\omega\big(\tau_{k}(\omega)\big)\Big)(t\wedge T)$ is measurable as a composition of the $\Bor(\R_+)\otimes\F$--measurable function
$(t,\omega)\mapsto \Big(t\wedge T,\tau_{k}(\omega),x_{k}(\omega),\omega\big(\tau_{k}(\omega)\big)\Big)$ with the Borel map $(t,a,y,j) \mapsto \Xi(a,y,j)(t)$ from $[0,T]\times [0,T]\times \R^N\times\ind$ to $\R^N$. The joint measurability of $(t,\omega)\mapsto \cchi_{(\tau_{k}(\omega),\tau_{k+1}(\omega)]}(t\wedge T)$ follows from the fact that
$\cchi_{(\tau_{k}(\omega),\tau_{k+1}(\omega)]}(t\wedge T)=\cchi_{F_k}(t,\omega)$ and
$$F_k:=\{(t,\omega)\in\R_+\times\Omega\,:\,\tau_k(\omega)<t\wedge T\leqslant \tau_{k+1}(\omega)\,\}\in\F.$$
As a countable sum of $\Bor(\R_+)\otimes\F$--measurable functions, we conclude that $\eta$ has the required measurability property.
\end{proof}

%
% Next, we want to prove that such a curve $\eta$ realizes the minimum for $(\S(T)\uu^0)_i(x)$ in the Lax--Oleinik formula \eqref{random Lax-Oleinik formula}, for every $i\in\ind$. We need a preliminary result first.

\smallskip

We proceed by showing a further property enjoyed by the curve $\eta$ defined above.

\begin{lemma}\label{lemma random minimization 1}
Let $\omega\in\Omega$. For $\leb^1$--a.e. $s\in (0,T)$ the following holds:
\begin{itemize}
 \item[(i)] \quad $t\mapsto \eta(t,\omega)$ is differentiable at $s$;\smallskip
 \item[(ii)] \quad $\displaystyle\lim\limits_{h\to 0^+}
 \frac 1 h \int_s^{s+h} L_{G_{\omega(t)}} \big(T-t,\eta(t,\omega),-\dot\eta(t,\omega)\big)\,\dd t
 =
 L_{G_{\omega(s)}} \big(T-s,\eta(s,\omega),-\dot\eta(s,\omega)\big)
 $;\smallskip

 \item[(iii)] \quad $t\mapsto u_{\omega(s)}\big(T-t,\eta(t,\omega)\big)$ is differentiable at $s$ and
 \begin{equation}\label{item 3}
  -\frac{\dd}{\dd t}u_{\omega(s)}\big(T-t,\eta(t,\omega)\big)_{\mbox{\Large $|$}_{t=s}}= L_{G_{\omega(s)}} \big(T-s,\eta(s,\omega),-\dot\eta(s,\omega)\big).
 \end{equation}
\end{itemize}
\end{lemma}

\begin{proof}
Let us fix $\omega\in\Omega$. The c\`{a}dl\`{a}g path
$t\mapsto\omega(t)$ has a finite number  of jump times  in $(0,T)$,
let us say $0<s_1<\dots<s_n<T$. Let us set $s_0=0$, $s_{n+1}=T$ and
pick $k\in\{0,\dots,n\}$. By definition of $\eta$, we have
$\eta(\cdot,\omega)=\xi(\cdot)$ in $[s_k,s_{k+1}]$, where
\[\xi(t)=\Xi\big(s_k,\eta(s_k,\omega),\omega(s_k)\big)(t)
\quad\hbox{ for every $t\in [0,T]$.}\] In view of Proposition
\ref{prop deterministic curves},  we have for every $t\in [s_k,T]$
\begin{eqnarray*}
u_{\omega(s_k)}\big(T-t,
\xi(t)\big)-u_{\omega(s_k)}^0\big(\xi(T)\big)=\int_{t}^{T}
L_{G_{\omega(s_k)}} \big(T-r,\xi(r),-\dot\xi(r)\big)\,\dd r.
\end{eqnarray*}
Choose  $s_k\leqslant t_0<t_1 \leqslant s_{k+1}$. By plugging
$t=t_0$ and $t=t_1$ in the above equality and by subtracting the
corresponding relations, we end up with
\begin{eqnarray}\label{eq calibrated}
 u_{\omega(s_k)}\big(T-t_0, \xi(t_0)\big) -  u_{\omega(s_k)}\big(T-t_1, \xi(t_1)\big)
 =
 \int_{t_0}^{t_1} L_{G_{\omega(s_k)}}\big(T-r,\xi(r),-\dot\xi(r)\big)\,\dd r.
\end{eqnarray}
By summing the equalities \eqref{eq calibrated} with $t_0=s_k$,
$t_1=s_{k+1}$ for $k=0,\dots, n+1$, we get
\begin{eqnarray}\label{eq sommabile}
\int_0^T L_{G_{\omega(t)}}\big(T-t,\eta(t,\omega),-\dot\eta(t,\omega)\big)\,\dd t \leqslant (n+1)\|\uu\|_{L^\infty([0,T]\times\R^N)}.
\end{eqnarray}
Since the functions $L_{G_i}$ are bounded from below, this tells us
that the map
\[t\mapsto L_{G_{\omega(t)}}
\big(T-t,\eta(t,\omega),-\dot\eta(t,\omega)\big)\] is integrable in
$[0,T]$. Therefore assertions (i) and (ii)  hold whenever $s$ is a
differentiability point of the curve $t\mapsto\eta(t,\omega)$ and a
Lebesgue point for $t\mapsto L_{G_{\omega(t)}}
\big(T-t,\eta(t,\omega),-\dot\eta(t,\omega)\big)$, namely for
$\leb^1$--a.e. $s\in (0,T)$. Plug $t_0 =s$ and $t_1 =s+h$ in
\eqref{eq calibrated} for any such point $s\in (0,T)$ and for $h>0$
small enough.
 By dividing the corresponding equality by $h$ and by passing to the limit, we finally obtain (iii).
\end{proof}

\medskip

\begin{proof}[Proof of Theorem \ref{butterfly}]
 We introduce  $h\in (0,T)$, devoted to become infinitesimal.  Since  $\uu(t,x):=
\S(t) \uu^0(x)$ is assumed locally Lipschitz continuous in $(0,T]
\times \R^N$,  $\uu$ is locally Lipschitz in $[h,T]\times\R^N$.

We can apply Theorem \ref{teo key} to the absolutely continuous
function $t \mapsto \EE_i
\big[u_{\omega(t)}\big(T-t,\eta(t,\omega)\big)\big]$ for $t\in
[h,T]$. By  taking into account Lemma \ref{lemma random minimization
1} and the definition of the functions $L_{G_j}$ we get
\begin{align}
   & \EE_i \big[u_{\omega (h)}\big(T-h,\eta(h,\omega)\big)\big] - \EE_i \big[u^0_{\omega (T)}\big(\eta(T,\omega)\big)\big]
   \label{butterfly0}\\
   &=- \int_h^T  \EE_i \left[-  (B\uu)_{\omega(s)}\big(T-s,\eta(s,\omega)\big) + \frac \dd{\dd t} u_{\omega(s)}\big(T -t,
\eta(t,\omega)\big)_{\mbox{\Large $|$}_{t=s}}\right]\, \dd s \nonumber\\
&= \int_h^T \EE_i \big[L_{\omega(s)}\big(\eta(s,\omega) , -
\dot\eta(s,\omega) \big)\big] \, \dd s = \EE_i \left [ \int_h^T
L_{\omega(s)}\big(\eta(s,\omega), - \dot\eta(s,\omega) \big) \, \dd
s \right ] . \nonumber
\end{align}
By sending  $h\to 0^+$, we get
\begin{equation}\label{butterfly1}
\lim_{h\to 0^+}
u_{\omega(h)}\big(T-h,\eta(h,\omega)\big)=u_{\omega(0)}(T,x)
    \quad\hbox{ for every
     $\omega\in\Omega$.}
\end{equation}
Moreover, since $\uu$ is bounded in $\clTcyl$, we obtain via
Dominated Convergence Theorem
\begin{equation}\label{butterfly2}
    \lim_{h\to 0^+}  \EE_i \big[u_{\omega
(h)}\big(T-h,\eta(h,\omega)\big)\big] = \EE_i \big[u_{\omega
(0)}\big(T,x\big)\big] = u_i(T,x).
\end{equation}

Further,  being the Lagrangians $L_j$  bounded from below, we get
via a standard application of the Monotone Convergence Theorem
\begin{align*}
\lim_{h\to 0^+}  \EE_i &  \left [ \int_h^T  L_{\omega(s)}\big(\eta(s,\omega),  - \dot\eta(s,\omega) \big) \, \dd s \right ]\\
&=
\lim_{h\to 0^+}  \int_0^T \EE_i \left [L_{\omega(s)}\big(\eta(s,\omega), - \dot\eta(s,\omega) \big) \right ] \,\cchi_{[h,T]}(s)\, \dd s\\
=
\int_0^T \EE_i &\left [L_{\omega(s)}\big(\eta(s,\omega), - \dot\eta(s,\omega) \big) \right ] \, \dd s
=
\EE_i \left [ \int_0^T L_{\omega(s)}\big(\eta(s,\omega), - \dot\eta(s,\omega) \big) \, \dd s \right ] .\qquad
\end{align*}

By putting together the above relation plus \eqref{butterfly0},
\eqref{butterfly1}, \eqref{butterfly2}, we have
\[
u_i(x,T)- \EE_i \big[u^0_{\omega (T)}\big(\eta(T,\omega)\big)\big]=
\EE_i \left [ \int_0^T L_{\omega(s)}\big(\eta(s,\omega), -
\dot\eta(s,\omega) \big) \, \dd s \right ] ,
\]
which shows the claimed minimality property of $\eta(t,\omega)$.
\end{proof}

\end{section}

\section{Properties of minimizing random curves}\label{regolare}
In the final section, we want to establish further properties of
arbitrary minimizing curves and of solutions of the evolution
equation. We start showing that  any minimizing curve has a similar
structure as the one constructed in the previous section, up to a
set of negligible probability.

We  consider  a solution $\uu$ of \eqref{HJS} in $(0,+ \infty)
\times \R^N$ taking an initial value $\uu^0$ bounded  and Lipschitz
continuous in $\R^N$.   The function $\uu$ is consequently Lipschitz
continuous in $[0,T]\times\R^N$, for any $T >0$, by Theorem \ref{teo
existence evo wcs}. We fix $T>0$, $x \in \R^N$, $i\in \ind$, and
denote by $\eta : \Omega\to\curves$
 an admissible curve realizing the minimum for $u_i(x,T)=\big (\S(T)\uu^0 \big )_i(x)$.
These notations will stay in place throughout the section.

\smallskip

\smallskip

\begin{lemma}\label{minimizer}
There  is a full measure set $\Omega'_i\subset \Omega_i$ such that
for all $\omega\in \Omega'_i$, if $ a <  b \in [0,T]$  such that
$\omega$  is constantly equal to some $j\in \ind$ in  $[ a, b)$,
then
$$
u_j\big(b, \eta(T-b,\om)\big) - u_j\big(a, \eta(T-a,\om)\big)
=\int_a^b L_{G_j}\big( s, \eta(T-s,\omega),-
\dot\eta(T-s,\omega)\big)\dd s.
$$
\end{lemma}

\begin{proof}
Using that we have equality in the proof of Theorem \ref{pro
dominationbis}, one concludes from  \eqref{key00} and
\eqref{Fenchel} that for a.e. $\omega$ and $s$,
$$\big(B\uu\big)_{\omega(s)}\big(s,\eta(T-s,\omega)\big)  + \frac{\dd}{\dd t}
u_{\omega(s)}\big(t,\eta(T-t,\omega)\big)_{\mbox{\Large $|$}_{t=s}}
= L_{\omega(s)}\big(\eta(T-s,\omega),- \dot\eta(T-s,\omega)\big).$$
It follows by Fubini's theorem that there exists a set
$\Omega_i'\subset \Omega_i$ such that for all $\omega\in \Omega'_i$,
the above relation holds for almost every $s\in [0,T]$. By
integrating,  for  $\omega\in \Omega'_i$, if $0<a<b<T$ are such that
$\om$ is constantly equal to $j$ on $[a,b)$, then
\begin{align*}
u_j\big(b, \eta(T-b,\om)\big) - u_j\big(a, \eta(T-a,\om)\big) &= \int_a^b \Big[-(B\uu)_j \big(s, \eta(T-s,\om)\big) \\
&\quad+ L_j\big(  \eta(T-s,\omega),- \dot\eta(T-s,\omega)\big)\Big]\dd s\\
&=\int_a^b L_{G_j}\big( s, \eta(T-s,\omega),-
\dot\eta(T-s,\omega)\big)\dd s.
\end{align*}
\end{proof}

\smallskip

\begin{oss} Notice that for the particular minimizing curve
constructed in the previous section the exceptional negligible set
is empty.
\end{oss}

\smallskip

When the Hamiltonians enjoy stronger regularity properties we will
accordingly get  further regularity information on the minimizing
curves as well as on the solutions on such curves.

 We assume in the remainder of the section  $H_1,\dots,H_m$ to satisfy, besides (H1)--(H3), the following
further assumptions:
\begin{itemize}
\item[(H4)] \quad $p\mapsto H(x,p)\qquad\hbox{is strictly
convex on $\R^N$ for any $x\in \R^N$}$;\medskip
\item[(H5)] \quad $H\in\D{C}^1(\R^N\times\R^N)$.\medskip
\end{itemize}
In this case, the associated Lagrangian $L$ is of class
$\D{C}^1$ in $\R^\N\times\R^N$ and is strictly convex in $q$. Moreover, the map $(x,q)\mapsto
\big(x,\partial_q L(x,q)\big)$ is a homeomorphism of $\R^N\times\R^N$ onto itself, with continuous
inverse given by $(x,p)\mapsto \big(x,\partial_p H(x,p)\big)$, see for instance \cite[Appendix A.2]{CaSi00}.

\smallskip

\begin{teorema}\label{principal} For any fixed $\omega \in \Omega_i'$ we
have
\begin{itemize}
    \item [(i)] the  curve $\eta(\cdot, \omega)$ is continuously differentiable in $(0,T)$ outside the jump times of $\omega$;
    \item [(ii)] if $t$ is a jump time of $\omega$ then $\eta(\cdot, \omega)$ is right and left--differentiable at $t$ with
    \[ \lim_{s \to t^+} \dot\eta(s,\omega) = \frac{\dd^+}{\dd \hbox{\em t}} \eta(t, \omega)
     \quad\hbox{and} \quad    \lim_{s \to t^-} \dot\eta(s, \omega) = \frac{\dd^-}{\dd \hbox{\em t}}\eta(t),\]
    where $\frac{\dd^+}{\dd \hbox{\em t}}$ and $\frac{\dd^-}{\dd \hbox{\em t}}$ indicates right
    and left derivatives, respectively;
    \item [(iii)] $\eta(\cdot, \omega)$ is right differentiable at $0$ and left differentiable at $T$ with
\[ \lim_{s \to 0^+} \dot\eta(s,\omega) = \frac{\dd^+}{\dd \hbox{\em t}} \eta(0,\omega)
\quad\hbox{and} \quad    \lim_{s \to T^-} \dot\eta(s,\omega) =
\frac{\dd^-}{\dd \hbox{\em t}}\eta(T,\omega).\]
\end{itemize}
\end{teorema}
\begin{proof}
  Let $\om \in \Omega'_i$, and $0\leqslant \hat a <\hat b\leqslant T$ such that $\om$ is constantly equal to $j\in \ind $ on $[\hat a , \hat b]$. By recalling that $u_j$ is a solution of \eqref{eq G} in $(0, + \infty) \times \R^N$
  and exploiting \eqref{eq Lax-Oleinik GG},
 we get that the curve $s\mapsto \eta(s,\om)$ is a minimizer of
\begin{equation*}
\xi \mapsto u_j\big(T-{\hat b}, \xi({\hat b})\big)-u_j\big(T-\hat a, \xi(\hat a)\big)+
\int_{\hat a}^{\hat b} L_{G_j}\big(T-s, \xi(s),- \dot\xi(s)\big)\dd s
\end{equation*}
on the space of absolutely continuous curves $\xi:[{\hat a},{\hat
b}]\to\R^N$ taking the value $\eta(\hat a, \omega)$ at $t = \hat a$.
Theorem 18.1 in \cite{Cl13} thus establishes, among other things,
that the map
\begin{equation}\label{rara1}
    t \mapsto \partial_q L_j\big(\eta(t,\omega), - \dot\eta(t,\omega)\big)
\end{equation}
which in principle is defined for a.e. $t \in [{\hat a},{\hat b}]$, can be extended to
 an absolutely continuous curve on $[{\hat a},{\hat b}]$, denoted by $p(\cdot)$.
 Next, we use  the regularity assumptions on the Hamiltonians to invert the relation in \eqref{rara1} and get
 \begin{equation}\label{rara2}
    -\dot \eta (t, \omega) = \partial_p H_j\big(\eta(t,\omega), p(t)\big)
 \qquad\hbox{for a.e. $t\in [{\hat a},{\hat b}]$.}
\end{equation}
  From the continuous character of $\eta(\cdot,\omega)$ and $p(\cdot)$,
it follows that $\dot\eta(\cdot,\om)$  can be continuously extended on $[{\hat a},{\hat b}]$. We deduce that
$\eta(\cdot, \omega)$ is Lipschitz continuous in $[\hat a,\hat b]$ and, by the above continuity properties of $\dot\eta$,
is  in addition  continuously differentiable in  $({\hat a},{\hat b})$.
This  gives  (i). We moreover have
\[ \frac {\eta(t,\om)- \eta(\hat a, \om)}{t - \hat a}= \frac 1{t - \hat a} \, \int_{\hat a }^t \dot\eta(s) \, \dd s\]
for $t \in (\hat a, \hat b)$.  Taking into account that $\dot\eta(\cdot,\om)$ is continuous in $(\hat a,\hat b)$
and can be continuously extended up to the boundary, we deduce that
\[  \frac{\dd^+}{\dd \hbox{\em t}} \eta(\hat a, \omega)= \lim_{ t \to {\hat a}^+} \frac {\eta(t,\om)- \eta(\hat a, \om)}{t - \hat a}=
 \lim_{t \to {\hat a}^+} \dot\eta(t). \]
 The above argument, with obvious adaptations, gives items (ii) and (iii), and concludes the proof.
\end{proof}

\smallskip

\begin{cor} For any fixed $\omega \in \Omega_i'$,
the function $u_{\omega(t)}$ is differentiable at $\big(T-t, \eta(t,\omega)\big)$, whenever $t$ is not a jump time
    of $\omega$ in $(0,T)$ and
\begin{align}
&\partial_x u_{\omega(t)}(T-t,\eta(t,\omega)\big)=\partial_q L_{\omega(t)}\big(\eta(t),-\dot\eta(t,\omega)\big) \label{derivee bornee}\\
&\partial_t u_{\omega(t)}(T-s,\eta(t,\omega)\big)=-H_{\omega(t)}\big(\eta(t,\omega),D_x u_{\omega(t)}\big(T-t,\eta\big(t,\omega)\big)\big)\\
&\qquad\qquad\qquad\qquad\qquad\qquad\qquad\qquad\qquad\qquad\qquad-(B\,\uu)_{\om(t) }\big(T-t,\eta(t,\omega)\big). \nonumber
\end{align}
Moreover if $t$ is a jump time, the same holds by replacing $\omega(t)$ with $\omega(t^-)$ and $\dot\eta$ with $\frac{\dd^-}{\dd \hbox{\em t}}\eta$.
\end{cor}
\begin{proof}
The assertion  directly comes from the previous result and Corollary \ref{corrrr} with $t =T-\hat a$,
 $a =T-\hat b$, $\gamma(s)=\eta(T-s,\omega)$ for $s \in[a,t]$,  and by making the
 change of variables from $s$ to  $\tau =T-s$ in the integral appearing in
 the representation formula of $u_{\omega (t)}$, see \eqref{eq Lax-Oleinik GG}.
\end{proof}

\medskip

In the sequel, we will  denote by $\D D \left (0,T;\R^N\right )$ the
Polish space of c\`adl\`ag paths taking values in $\R^N$, endowed with
the Prohorov metric, see \cite{Bill99}.

Keeping in mind Theorem \ref{principal}, we extend $\dot
\eta(\cdot,\omega)$ on the whole $[0,T]$ setting
\[\dot\eta(t, \om) = \left \{\begin{array}{cc}
                      \frac{\dd^+}{\dd \hbox{\em t}}\eta(t,\omega)   & \quad\hbox{if $t$ is a jump time of $\omega$ or $t =0$} \\
                       \frac{\dd^-}{\dd \hbox{\em t}}\eta(t,\omega)  & \quad\hbox{if $t =T$.} \\
                     \end{array} \right .\]

\smallskip

We further define for $t \in [0,T]$ the adjoint curve
\[ P(t,\om)= \partial_q L_{\omega(t)}\big(\eta(t,\omega), - \dot\eta(t,\omega)\big).\]

Note that thanks to \eqref{derivee bornee}, $P(t,\om) \in \partial^C_x u_{\om(t)}
 \big(T-t,\eta(\omega,t)\big) $ for all $t$ and $\om \in \Omega'_i$, where $\partial^C$ stands for the Clarke generalized grandient.
\smallskip

We deduce from the proof of Theorem \ref{principal}

\begin{cor}\label{postprincipal} For any fixed $\omega$, the curve $P(\cdot, \omega)$
is absolutely continuous on intervals  of $[0,T]$ where $\omega$ is constant.
\end{cor}

\smallskip
\begin{prop} The maps $\omega \to \dot\eta(\cdot, \omega)$, $\omega \to P(\cdot,\omega)$ are nonanticipating
random variables from $\Omega$ to $\D D \left (0,T;\R^N\right )$.
In addition, the jump times of $\dot\eta(\cdot, \omega)$, $P(\cdot,\omega)$ and $ \omega$ coincide,
for any $\omega \in \Omega_i'$, with the possible exception of $T$.
\end{prop}
\begin{proof} For any $\omega$, the curves $\dot\eta(\cdot, \omega)$, $P(\cdot, \omega)$ are c\`{a}dl\`{a}g
 by construction, with discontinuity points corresponding to the jump times of $\omega$, with the possible exception of $T$
 where $\dot\eta(\cdot, \omega)$ and $P(\cdot, \omega)$ are continuous.

 Thanks to Proposition \ref{dotga},
 $\dot \eta$ is in addition $\Bor([0,t])\otimes\F_t$--progressively measurable, for $t \in [0,T]$, which is, due to its
 c\`{a}dl\`{a}g character, is equivalent of being nonanticipating. The measurability properties of $\eta$, $\dot\eta$ and
 the fact that $\partial_q L$ is continuous in both arguments implies that $P$ is a random variable. It also inherits the nonanticipating
 character of $\eta$, $\dot\eta$.
\end{proof}

We know, thanks to Corollary \ref{postprincipal}, that
$P(\cdot,\omega)$ is a.e. differentiable in $[0,T]$, for any fixed
$\omega$. We  derive from \cite[Theorem 18.1]{Cl13} that it
satisfies a suitable differential  inclusion in $[0,T]$. Combining
this information with \eqref{rara1}, \eqref{rara2} and the very
definition of $P$, we moreover get that the pair
$\big(\eta(\cdot,\omega), P(\cdot, \omega)\big)$ is a trajectory of a
twisted generalized Hamiltonian dynamics, where the equation related
to $P$ is multivalued and contains a coupling term.
\smallskip

\begin{cor}\label{postpostprincipal} Given $\omega \in \Omega$, we have
\[ -\dot\eta(t, \omega)= \partial_p H_{\om(t)}\big(\eta(t,\omega),
P(t,\omega) \big)\quad\hbox{for any $t \in (0,T)$, not jump time of
$\omega$}\]
 and
\begin{eqnarray*}
  \dot P(t, \omega) &\in&  - \partial_x H_{\om(t)}
 \big(\eta(\omega,t), P (\omega,t)\big) -  \sum_{j=1}^m b_{\omega(t)j} \, \partial^C_x u_j
 \big(T-t,\eta(\omega,t)\big) \\
  &=&  \partial_x L_{\om(t)}
 \big(\eta(\omega,t), -\dot \eta (\omega,t)\big) -  \sum_{j=1}^m b_{\omega(t)j} \, \partial^C_x u_j
 \big(T-t,\eta(\omega,t)\big)
\end{eqnarray*}
 for a.e. $t \in [0,T]$, where $\partial^C_x u_j$ indicates the Clarke generalized gradient of $u_j$
 with respect to the state variable. The multivalued linear
 combination in the formula must be understood in the sense of \eqref{minko}.
\end{cor}

By combining Corollaries \ref{postprincipal}, \ref{postpostprincipal} and the continuity properties of $\eta$,
$\dot\eta$, we further get
\smallskip

\begin{cor} Given $\omega$, the curve $P(\cdot,\omega)$ is Lipschitz continuous on intervals
of $[0,T]$ where $\omega$ is constant.
\end{cor}

We conclude the section by showing that when the Hamiltonians are of
Tonelli type, the Lax--Oleinik semigroup has a regularizing effect,
similar to the one well known for scalar Hamilton--Jacobi equations.

Given an open convex set $U \subset \R^N$ and $C >0$, we recall that
a function $f: U \to \R$ is said semiconcave with semiconcavity
constant $C$ if
\[f(\lambda \, x + (1 - \lambda) \, y \geq \lambda \, f(x) + (1-\lambda) \, f(y)
- C \, |x-y|^2\]

\smallskip

\begin{teorema}
Assume, in addition to conditions (H1)--(H5), that the Hamiltonians
$H_1,\dots,H_m$ are  of class $\D C^2$ in $\R^N\times\R^N$ with
positive definite Hessian. Then, for any fixed  $t>0$, the function
$\uu(t,\cdot)$ is
 locally semiconcave with linear modulus on $ \R^N$.
\end{teorema}
\begin{proof}
We first remark that, under the conditions assumed on the
Hamiltonians, the associated Lagrangians are locally semiconcave in
$(x,q)$  with a linear modulus, see \cite{Fa}. Let $t>0$ and
$i\in\ind$ be fixed. We want to prove that the function
$u_i(t,\cdot)$ is locally semiconcave with linear modulus. Indeed,
consider $\eta $ realizing the minimum in the Lax--Oleinik formula
\eqref{random Lax-Oleinik formula} for $u_i(t,x)$, for some $x\in
\R^N$.

Given $z\in \R^N$, we define
\[ \eta_\pm (s,\om) = \eta(s,\om) \pm \frac{t-s}{t} z.\]
Those two curves are admissible and start at $x\pm z$. We can
estimate
\begin{align*}
u_i(t,x+z) &+ u_i(t,x-z) - 2u_i(t,x)  \\
\leqslant \EE_i\Big [
& u^0_{\omega(t)}\big(\eta_+(t,\omega)\big)+\int_0^t
  L_{\omega(s)}\big(\eta_+(s,\omega),-\dot\eta_+(s,\omega)\big)\,\dd s\\
  + & u^0_{\omega(t)}\big(\eta_-(t,\omega)\big)
+ \int_0^t L_{\omega(s)}\big(\eta_-(s,\omega),-\dot\eta_-(s,\omega)\big) \,\dd s \\
 -& 2u^0_{\omega(t)}\big(\eta(t,\omega)\big) -2\int_0^t L_{\omega(s)}\big(\eta(s,\omega),-\dot\eta(s,\omega)\big)\,\dd s
  \Big] \\
 =   \EE_i\Big [ &\int_0^t \Big[L_{\omega(s)}\big(\eta(s,\omega)- \frac{t-s}{t} z,-\dot\eta(s,\omega)-\frac1t z\big) \\
 &\quad + L_{\omega(s)}\big(\eta(s,\omega)+ \frac{t-s}{t} z,-\dot\eta(s,\omega)+\frac1t z\big) \\
 &\quad -2 L_{\omega(s)}\big(\eta(s,\omega),-\dot\eta(s,\omega)\big)\Big]\,\dd s      \Big]\\
 \leqslant &  \EE_i\Big [  \int_0^t 2 \, C\frac{(t-s)^2+1}{t^2}|z|^2 \; \dd s \Big]= 2 \, C\Big(\frac t3 + \frac 1t\Big)|z|^2,
\end{align*}
where $C$ is a constant of semiconcavity of the $L_j$, $j\in \ind$ restricted to a neighborhood of the
 $(\eta,\dot \eta)$ and $(\eta_\pm,\dot \eta_\pm)$ that  are relatively compact by  \eqref{derivee bornee} (uniformly with respect to $\om$).

% {\color{blue} Il calcolo mi pare giusto a parte che mi viene  un fattore $2$ in pi\`{u},
%  non capisco per\`{o} perch\`{e} la $\dot\eta$ debba
% essere  limitata  in $[0,t]$ uniformemente rispetto a $\omega$. Essendo  a valori c`adl`ag
% \`{e} sicuramente limitata in $[0,t]$ per ogni fissato $\omega$, ma
% a come si fa a  dire di pi\`{u} ? Non vedo come
% si possa concludere la dimostrazione. Antonio}
\end{proof}

\begin{appendix}
\begin{section}{PDE material}\label{appendix PDE}
\subsection{For systems}
This section is devoted to the proofs of Theorem \ref{teo existence evo wcs} and Proposition \ref{prop comparison}. We prove a preliminary comparison result first:

\begin{prop}\label{prop Lip comparison}
Let $T>0$ and $\vv,\,\ww:[0,T]\times\R^N\to\R^m$ be a bounded upper
semicontinuous subsolution and a bounded lower semicontinuous
supersolution of \eqref{HJS}, respectively. Let us furthermore assume that either
$\vv$ or $\ww$ are in $\tvLip$.
Then
\[
 v_i(t,x)-w_i(t,x)\leqslant\max_{1\leqslant i \leqslant m}\,\sup_{\R^N} \big(v_i(0,\cdot)-w_i(0,\cdot)\big)
\]
for all $(t,x)\in\ccyl$ and $i\in\ind$.
\end{prop}

\begin{proof}
The result essentially follows from Proposition 3.1 in
\cite{CamLey}, which covers our case under the additional assumption
that there exists a continuity modulus $\widetilde\nu$ such that
\begin{eqnarray}\label{eq app pde condition}
\max_{i\in\ind}|H_i(x,p)-H_i(y,p)|\leqslant
\widetilde\nu\big((1+|p|)|x-y|\big)\quad\hbox{for all
$(x,p)\in\R^N\times\R^N$.}
\end{eqnarray}
When either $\vv$ or $\ww$ is Lipschitz continuous in $\clTcyl$, such hypothesis can be safely removed. Indeed, with the notation used in \cite{CamLey}, we see that either $\overline p +2\beta\overline x$ or $\overline p -2\beta\overline y$ is bounded,  uniformly with respect to the parameters $\alpha, \beta, \eta, \mu$, since it belongs to the super differential of $v_{\overline j}(\overline t,\cdot)$ at $\overline x$ or to the subdifferential of $w_{\overline j}(\overline s,\cdot)$ at $\overline y$. Using the estimates (3.3) and (3.4) in \cite{CamLey}, we conclude that both $\overline p +2\beta\overline x$ and $\overline p -2\beta\overline y$ are bounded, uniformly with respect to the parameters, and the result follows without invoking condition \eqref{eq app pde condition}.
\end{proof}

We now proceed by proving the existence part in the statement of Theorem \ref{teo existence evo wcs}.\medskip

\noindent{\em Proof of Theorem \ref{teo existence evo wcs} (Existence of solutions).} Let us first assume $\uu^0$ bounded and Lipschitz on $\R^N$ and fix $T>0$. Let us denote by $b_\infty:=\max_i\sum_{j=1}^m |b_{ij}|$ \ and pick a constant $C$ such that
\begin{equation}\label{app big C}
C>\max_{i\in\ind} \|H_i(x,Du^0_i(x))\|_{\infty}+b_\infty\|\uu^0\|_\infty.
\end{equation}
Set \ $M:=C+b_\infty(\|\uu^0\|_\infty+CT)$ and choose $n\in\N$ large enough so that the Hamiltonians
\[
 \widetilde H_i(x,p):=\min\left\{H_i(x,p),|p|+n \right\}\qquad\hbox{$(x,p)\in\R^N\times\R^N$ and $i\in\ind$}
\]
satisfy
\begin{eqnarray}\label{app modified Ham}
\widetilde H_i =H_i\quad\hbox{on $\big\{(x,p):\, \max_i \widetilde H_i(x,p)< M+1\,\big\}$}
\qquad\hbox{for all $i\in\ind$.}
\end{eqnarray}
The modified Hamiltonians $\widetilde H_i$ satisfy the additional condition \eqref{eq app pde condition}, thus
we can apply Proposition 3.1 in \cite{CamLey} and infer the existence of a function $\uu\in\tvBUC$ which solves \eqref{HJS} in $\Tcyl$ with $\widetilde H_i$ in
place of $H_i$ and satisfying the initial condition $\uu(0,\cdot)=\uu^0$ on $\R^N$.

We shall now prove that $\uu$ is Lipschitz in $\clTcyl$. To this aim, first notice that the functions $\uu^+(t,x):=\uu^0(x)+tC\1$ and  $\uu^-(t,x):=\uu^0(x)-tC\1$ are a Lipschitz super and subsolution to \eqref{HJS} in $\Tcyl$ with $\widetilde H_i$ in place of $H_i$. Therefore, by Proposition 3.1 in \cite{CamLey}, we get in particular
\begin{equation}\label{app t-Lip estimate}
|u_i(t,x)-u^0_i(x)|\leqslant Ct\qquad\hbox{for every $(t,x,i)\in\clTcyl\times\ind$.}
\end{equation}
By applying Proposition 3.1 in \cite{CamLey} again to $\uu(h+\cdot,\cdot)$ and to the solutions $\ww^\pm:=\uu\pm \|\uu(h,\cdot)-\uu^0\|_\infty$, we get
\[
\|\uu(t+h,\cdot)-\uu(t,\cdot)\|_{\infty}
\leqslant
\|\uu(h,\cdot)-\uu^0\|_{\infty}
\leqslant
C\,h\qquad\hbox{for every $h>0$.}
\]
This shows that the function $\uu$ is $C$--Lipschitz in $t$. By making use of the fact that $\uu$ is a solution to \eqref{HJS} in $\Tcyl$ with $\widetilde H_i$ in
place of $H_i$ and of the estimate \eqref{app t-Lip estimate}, we get
\[
 \widetilde H_i(x,D_xu_i(t,x))\leqslant
 -\partial_t u_i(t,x)-\big(B\uu(t,x)\big)_i
 \leqslant
 C+b_\infty(\|\uu^0\|_\infty+CT)=M
 \]
in the viscosity sense in $\Tcyl$. By coercivity of $\widetilde H_i$ in $p$, we infer that $u_i(t,\cdot)$ is Lipschitz for every $t\in (0,T)$ with $(x,D_x u_i(t,x))\in\{\widetilde H_i(x,p)\leqslant M\}$ for a.e. $x\in\R^N$, see for instance Lemma 2.5 in \cite{barles_book}. In view of \eqref{app modified Ham}, this finally implies that $\uu$ is a solution of \eqref{HJS} in $\Tcyl$ as well.

Let now assume that $\uu^0\in\vBUC$. Let $\mathbf{g}^n$ be a sequence of Lipschitz functions in $\vBUC$ uniformly converging to $\uu^0$ on $\R^N$, and denote by $\uu^n$ the corresponding bounded and Lipschitz solution to \eqref{HJS} in $(0,T)\times\R^N$ with initial datum $\mathbf{g}^n$.
By Proposition \ref{prop Lip comparison} we have\ $\|\uu^m-\uu^n\|_{L^\infty(\clTcyl)}
\leqslant
\|\mathbf{g}^m-\mathbf{g}^n\|_{L^\infty(\R^N)},$\
that is, $(\uu^n)_n$ is a Cauchy sequence in $\clTcyl$ with respect to the sup--norm. Hence the Lipschitz continuous functions $\uu^n$ uniformly converge to a function $\uu$ on $\clTcyl$, which is therefore bounded and uniformly continuous on $\clTcyl$. By the stability of the notion of viscosity solution, we conclude that $\uu$ is a solution of \eqref{HJS} with initial datum $\uu^0$. This completes the proof since $T>0$ was arbitrarily chosen.
\qed
\vspace{2ex}

The uniqueness part in Theorem \ref{teo existence evo wcs} is guaranteed by the comparison principle stated in Proposition \ref{prop comparison}, that we prove now. The proof makes use of the existence result just established and of Proposition \ref{prop Lip comparison}.\\

\noindent{\em Proof of Proposition \ref{prop comparison}.} Up to trivial cases and up to adding a constant vector of the form $C\1$ to $\ww$, we reduce the assertion to proving that $\vv\leqslant \ww$ in $\clTcyl$ for every fixed $T>0$ when $\vv(0,\cdot)\leqslant \ww(0,\cdot)$ in $\R^N$. Let us fix $T>0$ and, for every $\eps>0$,  set $\ww^\eps:=\ww+\eps\1$. Since either $\vv(0,\cdot)$ or $\ww^\eps(0,\cdot)$ are in $\vBUC$ and $\ww^\eps(0,\cdot)-\vv(0,\cdot)\geqslant \eps\1$, we can find $\uu^0\in\vBUC \cap\vLip$ such that \ $\vv(0,\cdot)\leqslant \uu^0\leqslant \ww^\eps(0,\cdot)$\ in $\R^N$. According to the existence part of Theorem \ref{teo existence evo wcs} proved above, we know that there exists a Lipschitz function $\uu\in\tvBUC$ which solves \eqref{HJS} in $\Tcyl$ with initial datum $\uu_0$. By applying Proposition \ref{prop Lip comparison} to the pair $\vv,\,\uu$ and $\uu,\,\ww^\eps$, respectively, we get \ $\vv\leqslant \uu\leqslant \ww^\eps$\ \ in $\clTcyl$. The assertion follows by sending
$\eps\to 0^+$. \qed

\vspace{2ex}

\subsection{For a single equation}
We now turn back to results concerning a single equation.
Let  $G:[0,T]\times\R^N\times\R^N\to\R$ be a continuous Hamiltonian such that, for every fixed $(t,x)\in[0,T]\times\R^N$, $G(t,x,\cdot)$ is convex in $\R^N$, and there exist two superlinear functions $\alpha,\beta:\R_+\to\R$ such that
\[
\alpha(|p|)\leqslant {G(t,x,p)}\leqslant \beta(|p|)$\qquad for all $(t,x,p)\in [0,T]\times\R^N\times\R^N.
\]
We will denote by $L_G:[0,T]\times\R^N\times\R^N\to\R$ the Lagrangian associated with $G$ through the Fenchel transform.
The following holds:

\begin{prop}\label{L.O.t}
Let $u:[0,T]\times\R^N\to\R$ be a Lipschitz  solution of
\begin{equation}\label{app eq G}
\frac{\partial u}{\partial t} + G(t,x,D_x u) = 0\qquad\hbox{in $(0,T)\times\R^N$}.
\end{equation}
Then for every $0\leqslant a<t\leqslant T$, the following identity holds:
\begin{equation}\label{eq Lax-Oleinik G}
u(t,x)=\inf_{\xi(t)=x}\Big( u\big(a,\xi(a)\big)+\int_a^t L_G\big(s,\xi (s), \dot \xi (s) \big) \dd s\Big),\qquad x\in\R^N,
\end{equation}
where the infimum is taken by letting $\xi$ vary in the family of absolutely continuous curves from $[a,t]$ to $\R^N$. Moreover, such an infimum is a minimum.
\end{prop}

%{\color{blue} ??? $u$ bounded non serve, ma serve Lipschitz, guardare bene il Thm. 3.12 di \cite{bardi} citato.}

\begin{proof}
Let us first prove the assertion for $a=0$. It is always true that
\begin{eqnarray}\label{app def value function}
u(t,x)\leqslant V(t,x):=\inf_{\xi(t)=x}\Big( u\big(0,\xi(0)\big)+\int_0^t L_G\big(s,\xi (s), \dot \xi (s) \big) \dd s\Big),
\end{eqnarray}
for every $x\in \R^N$.
Indeed, let $\xi:[0,t]\to\R^N$ be any absolutely continuous curve with $\xi(t)=x$. Then, for almost every $s\in [0,T]$, we have
\begin{eqnarray*}
\frac \dd{\dd \tau} u\big(\tau,\xi(\tau)\big)_{\mbox{\Large $|$}_{\tau=s}}
=
p_s+\langle p_{\xi(s)},\dot \xi(s) \rangle
\leqslant
p_s+G(s,\xi(s),p_{\xi(s)})+L_G\big(s,\xi(s),\dot\xi(s)\big),
\end{eqnarray*}
where $\big(p_s,p_{\xi(s)}\big)$ is a suitable element of the Clarke
generalized gradient of $u$ at $\big(s,\xi(s)\big)$, chosen according to
Lemma \ref{clarke}. By integrating the above inequality between $0$
and $t$ and by taking into account that $u$ is a subsolution of
\eqref{eq G}
 and $\xi$ was arbitrarily chosen, we readily get \eqref{app def value function}.

We therefore have to prove the converse inequality. Let us fix $R>0$ and define
$L_R = L_G$ on $[0,T]\times\R^N \times B_R$ and $+\infty$ elsewhere. Let us set
$$
V_R(t,x):= \inf_{\xi(t)=x}\Big( u\big(0,\xi(0)\big)+\int_0^t L_R\big(s,\xi (s), \dot \xi (s) \big) \dd s\Big),\qquad (t,x)\in (0,T]\times\R^N,
$$
where the infimum is taken by letting $\xi$ vary in the family of absolutely continuous curves from $[0,t]$ to $\R^N$. Clearly, it is the same to take the infimum over $R$--Lipschitz curves. Moreover, $V_R \geqslant V$.
As the curves take velocities in a compact set, by basic results of Optimal Control Theory (see \cite[Theorem 3.17 and Exercise 3.7]{bardi}), the function $V_R$ is a Lipschitz continuous  solution of
\begin{equation}\label{eq G_R}
\frac{\partial V_R}{\partial t} + G_R(t,x,D_x V_R) = 0\qquad\hbox{in $(0,T)\times\R^N$}
\end{equation}
where $G_R$ is the convex dual of $L_R$. It is easily checked from its very definition that $G_R$ is $R$--Lipschitz in $p$, and uniformly continuous in $[0,T]\times K\times \R^N$, for every compact set $K\subset \R^N$. Let us denote by $\kappa$ a Lipschitz constant of $u$ in $[0,T]\times \R^N$ and choose $R$ big enough
 such that $G=G_R$ on $[0,T]\times \R^N \times B_{\kappa+1}$. From the definition of (viscosity) solutions and the fact that sub and supertangents to $u$ have norms bounded by $\kappa$, it follows that $u$ also solves \eqref{eq G_R} in $[0,T] \times \R^N$. Now $u$ and $V_R$ are both Lipschitz solutions to \eqref{eq G_R} with same initial data, hence, by applying \cite[Theorem 3.12]{bardi}, we conclude that $u\equiv V_R$ in $[0,T]\times \R^N$. Since $V_R\geqslant V$, we finally get $u\equiv V$.

The fact that \eqref{eq Lax-Oleinik G} holds for $0\leqslant a<t\leqslant T$ is due to the fact that the function $V$ defined by \eqref{app def value function} satisfies the Dynamic Programming Principle. The fact that the infimum in \eqref{eq Lax-Oleinik G} is attained follows from classical results of the Calculus of Variations.
\end{proof}

We proceed by proving some differentiability properties of the solution $u$ at points belonging to the support of a minimizing curve for \eqref{eq Lax-Oleinik G}.

\begin{prop}\label{diffLO}
Let $G$ be as above, and assume moreover that $G(t,x,p)$ is strictly convex in $p$ and Lipschitz in $x$, locally with respect to $(t,x,p)$. Let $u$ be a Lipschitz solution of the evolutive Hamilton--Jacobi equation \eqref{eq G} in $[0,T]\times\R^N$.
Let $0<a< t<T$ and $\gamma:[a,t]\to\R^N$ a curve that realizes the infimum in \eqref{eq Lax-Oleinik G}. Assume $a$ is a differentiability point of $\gamma $ and a Lebesgue point of $s\mapsto L_G\big(s,\gamma (s),\dot\gamma (s)\big)$. Then $u$  is differentiable in $\big(a,\gamma (a)\big)$ and
\begin{eqnarray*}
\quad \partial_x u(a,\gamma (a)\big)=\partial_q L_G\big(a,\gamma (a),\dot\gamma (a)\big),\qquad \partial_t u(a,\gamma (a)\big)=-G\big(a,\gamma (a),D_x u\big(a,\gamma (a)\big)\big).
\end{eqnarray*}
\end{prop}

\begin{proof}
Under the above hypotheses, the function $u$ is locally semiconcave on $(0,+\infty ) \times \R^N$ (see \cite[Theorem 5.3.8]{CaSi00})  and $\gamma $ is Lipschitz. The proof is borrowed from \cite[Theorem 6.4.7]{CaSi00}. Let $(p_t,p_x) $ be a superdifferential to $u$ at $\big(a, \gamma (a)\big)$. By definition of (viscosity) solutions, $p_t +G(a,\gamma (a), p_x)\leqslant 0$. We will prove that equality holds. As for $t-a>h>0$, we have $u\big(a+h, \gamma (a+h)\big) - u\big(a,\gamma (a)\big) = \int_a^{a+h} L_G\big(s,\gamma (s), \dot \gamma (s) \big) \dd s$, since $a$ is a Lebesgue point of $s\mapsto L_G\big(s,\gamma (s),\dot\gamma (s)\big)$, it follows that
$$ \lim_{h\to 0^+}\frac{ u\big(a+h, \gamma (a+h)\big) - u\big(a,\gamma (a)\big)   }{h} = L_G\big(a,\gamma (a),\dot\gamma (a)\big).
$$
On the other hand, by properties of superdifferentials,
$$ \lim_{h\to 0^+}\frac{ u\big(a+h, \gamma (a+h)\big) - u\big(a,\gamma (a)\big)   }{h} \leqslant p_t + \langle p_x,\dot\gamma (a)\rangle.
$$
It follows from the Fenchel inequality that $p_t + G(a, \gamma (a),p_x) \geqslant 0$ hence the claimed equality holds.

Finally, as the superdifferential is convex and $G$ is strictly convex in the last argument, $D^+u \big(a,\gamma (a)\big)$ cannot contain more than one element. It is moreover not empty by properties of semiconcave functions, hence  $D^+u \big(a,\gamma (a)\big)$ is reduced to a singleton and $u$ is differentiable at $\big(a,\gamma (a)\big)$. Moreover, as $p_x$ realize the equality in the Fenchel inequality it follows that
$p_x = \partial_q L_G\big(a,\gamma (a),\dot\gamma (a)\big)$ and $p_t=-G\big(a,\gamma (a),D_x u\big(a,\gamma (a)\big)\big)$.

\end{proof}

We  finally state a consequence of the previous results and of Theorem 18.1 in \cite{Cl13}:
\begin{cor}\label{corrrr}
Let us assume that the hypotheses of Proposition \ref{diffLO} are in force, and furthermore that $G(t,x,\cdot)$ is of class $\D{C}^1$, for every fixed $(t,x)\in [0,T]\times\R^N$. Then the curve $\gamma $ is $\D{C}^1$ and $u$ is differentiable at $\big(s,\gamma (s)\big)$ for every $s\in [a,t)$, with
\begin{eqnarray*}
\quad \partial_x u(s,\gamma (s)\big)=\partial_q L_G\big(s,\gamma (s),\dot\gamma (s)\big),\qquad
\partial_t u(s,\gamma (s)\big)=-G\big(s,\gamma (s),D_x u\big(s,\gamma (s)\big)\big).
\end{eqnarray*}
\end{cor}

\end{section}

\end{appendix}

\bibliography{weakly}
\bibliographystyle{siam}

\end{document}